\documentclass[12pt]{amsart}
\usepackage{dsfont}
\usepackage{enumerate}
\usepackage{amssymb}
\usepackage{amsmath}
\usepackage[all]{xy}

\def\acts{\curvearrowright}

\DeclareMathOperator\Sc{s}
\DeclareMathOperator\C{\mathbb C}
\DeclareMathOperator\Z{\mathbb Z}
\DeclareMathOperator\N{\mathbb N}
\DeclareMathOperator\T{\mathbb T}
\DeclareMathOperator\Hom{\mathrm Hom}
\DeclareMathOperator\Gr{\mathrm Gr}
\DeclareMathOperator\I{\mathcal I}
\DeclareMathOperator\F{\mathcal F}

\DeclareMathOperator\codim{codim}
\DeclareMathOperator\SSS{\mathcal S}
\DeclareMathOperator\E{\mathds 1}
\DeclareMathOperator\A{\mathbb A}
\DeclareMathOperator\RES{RES}
\DeclareMathOperator\sym{Sym}

\DeclareMathOperator\PPP{\mathcal P}

\def\Res#1{\mathop{\mathrm{Res}}_{#1}}

\newcommand\Fcsm{{\F}^{csm}}
\newcommand\Fssm{{\F}^{ssm}}
\newcommand{\aalpha}{{\boldsymbol \alpha}}
\newcommand{\bbeta}{{\boldsymbol \beta}}
\newcommand{\zz}{{\boldsymbol z}}
\newcommand{\tssm}{ \tilde{\Sc} }
\newcommand{\csm}{ c^{sm}}
\newcommand{\ssm}{ s^{sm}}

\newtheorem{fact}{Fact}[section]
\newtheorem{lemma}[fact]{Lemma}
\newtheorem{theorem}[fact]{Theorem}
\newtheorem{definition}[fact]{Definition}
\newtheorem{example}[fact]{Example}
\newtheorem{rremark}[fact]{Remark}
\newenvironment{remark}{\begin{rremark} \rm}{\end{rremark}}
\newtheorem{proposition}[fact]{Proposition}
\newtheorem{corollary}[fact]{Corollary}
\newtheorem{conjecture}[fact]{Conjecture}

\author{L. M. Feh\'er}
\address{E\"otv\"os University, Budapest, Hungary}
\email{lfeher@cs.elte.hu}

\author{R. Rim\'anyi}
\address{Department of Mathematics, University of North Carolina at Chapel Hill, USA}
\email{rimanyi@email.unc.edu}

\title{Chern-Schwartz-MacPherson classes of degeneracy loci}

\begin{document}

\begin{abstract}
The Chern-Schwartz-MacPherson class (CSM) and the Segre-Schwartz-MacPherson class (SSM) are deformations of the fundamental class of an algebraic variety. They encode
finer enumerative invariants of the variety than its fundamental class. In this paper we offer three contributions to the theory of equivariant CSM/SSM classes. First, we
prove an interpolation characterization for CSM classes of certain representations. This method---inspired by recent works of Maulik-Okounkov and
Gorbounov-Rimanyi-Tarasov-Varchenko---does not require a resolution of singularities and often produces explicit (not sieve) formulas for CSM classes. Second, using the
interpolation characterisation we prove explicit formulas---including residue generating sequences---for the CSM and SSM classes of matrix Schubert varieties. Third, we
suggest that a stable version of the SSM class of matrix Schubert varieties will serve as the building block of equivariant SSM theory, similarly to how the Schur
functions are the building blocks of fundamental class theory. We illustrate these phenomena, and related stability and (2-step) positivity properties for some relevant
representations.
\end{abstract}

\maketitle

%\tableofcontents

\section{Introduction}

\subsection{Degeneraci loci, fundamental class, Schur expansion}

Many interesting varieties in geometry occur as {\em degeneracy locus} varieties, a notion we recall now. Let $\Sigma\subset V$ be an invariant variety of a
$G$-representation $V$. Let $E\to M$ be a vector bundle over a smooth variety $M$ with fiber $V$ and structure group $G$. If $\Sigma(E)$ is the  union of the $\Sigma$'s in
each fiber then the subvariety $X=\sigma^{-1}( \Sigma(E) )$ for a section $\sigma$ transversal to $\Sigma(E)$ is called a {\em degeneracy locus}. Areas of geometry where
degeneraci loci are abundant include Schubert calculus, moduli spaces, and singularity theory.

The general strategy of studying numerical invariants of degeneracy loci is associating a ``universal'' $G$-characterisitic class to the {\em local} situation $\Sigma\subset
V$, and expecting that the sought numerical invariant of $X$ is obtained by evaluating the universal characteristic class at the bundle $E \to M$. The key example of this
strategy is the {\em fundamental class}
$[X]\in H^{*}(M)$. One defines the $G$-equivariant fundamental class $[\Sigma]\in H^{*}_G(V)=H^{*}(BG)$, and it is a fact that $[X]\in H^*(M)$ can be calculated as
$[\Sigma]$ evaluated at the bundle $E\to M$.

Hence, equivariant fundamental classes $[\Sigma]\in H_G^*(V)$ and their applications have been intensively studied in numerous parts of geometry. Two %huge
 interesting sets of examples are (a) quiver representations (where the fundamental class is often called quiver polynomial), (b) singularity theory (where the fundamental
 class is called Thom polynomial). In the intersection of these two sets of examples is the Giambelli-Thom-Porteous formula \cite{porteous} for the fundamental class of the orbit closures of
 the representation $GL_k(\C)\times GL_n(\C)$ acting on $\Hom(\C^k,\C^n)$.

Fundamental classes in both of these sets of examples above show interesting patterns, namely stabilization and positivity properties. Stabilization properties are displayed
by the fact that the classes can be encoded by (generalized, so-called iterated residue) generating sequences, see e.g. \cite{bsz,noAss,rr}. Positivity means that the
coefficients in appropriate Schur expansions of the classes are non-negative, see e.g. \cite{buch,pw}.

\subsection{The Schwartz-MacPherson deformation of the fundamental class}

The notion of fundamental class $[X]\in H^{\codim (X \subset M)}(M)$ has a deformation \cite{M}, which comes in two versions called Chern-Schwartz-MacPherson (CSM) and
Segre-Schwartz-MacPherson (SSM) classes. The two versions only differ by an explicit factor. (Another name for the CSM class, after homogenization by a new variable $\hbar$,
is `characteristic cycle' class.) The CSM/SSM classes, $\csm(X), \ssm(X)$ are inhomogeneous elements of $H^*(M)$. The lowest degree part of both of them is the fundamental
class $[X]$. The CSM/SSM classes encrypt more intristic information about the variety $X$ than the fundamental class; and their applications in enumerative geometry is hard
to overestimate, see e.g. \cite[Ch. 5]{fultonbook}. P. Aluffi  writes {\em ``Segre classes provide a systematic framework for enumerative geometry computation; but this is
of relatively little utility, as Segre classes are in general hard to compute''} \cite{ACCSV}. More recent applications of CSM classes (under the name of ``stable envelope
classes'') are in \cite{MO, Ok, RTV2} and references therein.

Ohmoto \cite{O1,O2} showed that the above mentioned equivariant strategy works for the SSM class. One can associate a universal characteristic class $\ssm(\Sigma\subset V)$
to the situation $G\acts(\Sigma \subset V)$ and the SSM class of the degeneracy locus $X$ is this characteristic class evaluated at the defining bundle, see Theorem
\ref{thm:ohmoto} below. Hence, CSM/SSM theory of degeneracy loci is reduced to finding CSM/SSM classes for invariant varieties in {\em representations}. In the present paper
we offer three contributions to this problem, described in the next three subsections.

\subsection{Interpolation characterization of CSM classes in representations}
The classical approach to calculating either the fundamental class or the CSM/SSM class of $\Sigma\subset V$ (acted upon by the group $G$) is resolution of singularities.
Resolutions are not known for
important  examples of quivers or contact singularities. Even if a resolution is known, this method produces exclusion-inclusion type formulas that hide both the
stabilization and positivity structures  of CSM/SSM classes. In the past decades new and effective methods of calculating fundamental classes were found. One of these new
methods is interpolation \cite{rrinv, br}: one lists a few interpolation condition that $[\Sigma]$ and only $[\Sigma]$ satisfies.

In this paper we prove that interpolation characterisation exists for CSM classes of orbits of certain representations, see Theorem~\ref{thm:intchar} below. We expect that
this new interpolation method will open the way to find CSM/SSM classes for quivers, matroids, singularities; similarly to what happened for the fundamental class in the
past two decades.

Although formally we will not use it, let us comment on the the origin of this interpolation theorem. In works of  Maulik-Okounkov \cite{MO} and
Gorbounov-Rimanyi-Tarasov-Varchenko \cite{GRTV, RTV} two seemingly unrelated modules are identified: (i) the regular representation of the equivariant cohomology algebra of
certain symmetric spaces, and  (ii) the Bethe algebra of certain quantum integrable systems.
On both sides of this identification there is a ``given'' basis and a ``sought'' basis. On the physics side the coordinate basis (a.k.a. spin basis) is given, and one seeks
formulas for the Bethe basis. On the geometry side the basis of torus fixed points is given, and one seeks formulas for classes coming from Schubert varieties. Our
identification matches the given basis on one side with the sought basis of the other side. In particular, the coordinate basis of the Bethe algebra is matched with the CSM
classes of Schubert cells \cite{RV}.
Our Theorem~\ref{thm:intchar} stems from this fact, and it is a version of the main theorem of \cite{RV}, modified from Schubert calculus settings to representations with
finitely many orbits satisfying the Euler condition.

\subsection{CSM/SSM classes of matrix Schubert cells: weight functions, generating functions}
%Schur functions play an essential role in the theory of fundamental classes. Schur functions are themselves fundamental classes in at least two situations: they are the fundamental classes of Schubert varieties in Grassmannians, and they are the fundamental classes of matrix Schubert varieties in a certain representation. In fact, the fundamental class of a Schubert variety in a Grassmannian lives in a quotient ring, and the Schur function is just one natural representative of it. The fundamental class of a matrix Schubert variety lives in a polynomial ring, hence it completely determines the Schur function.

The building blocks of the algebraic combinatorics of fundamental classes for both quivers and for singularities are the Schur functions. Schur functions are the fundamental
classes of so-called matrix Schubert varieties \cite{ss,km}. In Sections \ref{sec:csmW}-\ref{sec:GFpartition} we calculate the CSM and SSM classes of matrix Schubert cells.
Our formulas are not of inclusion-exclusion type; rather, some of our formulas are of ``localization type'' (inherently displaying interpolation properties), and others are
iterated residue generating sequences (inherently displaying stabilization properties of their Schur expansions).

It is important to emphasize that the connection between the CSM classes of {\em ordinary} Schubert cells and CSM classes of matrix Schubert cells are more complex than for
the corresponding fundamental classes (which is the lowest degree part of the CSM class). In Appendix B, Section \ref{appB} we summarize the differences and relations
between the CSM theory of Schubert cells and matrix Schubert cells.

In Section \ref{sec:GFpartition} we make a conjecture about the signs of the Schur expansions of the SSM classes of matrix Schubert cells. We are not aware of a direct
relation between this conjecture and the positivity theorem on CSM classes of ordinary Schubert cells conjectured by Aluffi-Mihalcea \cite{AM1} (see also \cite{AM2}) and
proved by Huh \cite{huh}---for more comments see Appendix B, Section \ref{appB}.

\subsection{Conjectured two-step positivity of SSM classes}
As mentioned, the building blocks of cohomological fundamental class theory are the Schur functions: When a fundamental class of a geometrically relevant variety is expanded
in appropriate Schur functions, the coefficients are often non-negative, see e.g. \cite{buch, pw}.

Schur positivity (or alternating signs) of CSM/SSM classes break down in the simplest examples. For example, according to \cite{PP}, a certain SSM class of the $A_2$ quiver
representation is (for notations see Section \ref{sec:sym})
\begin{align}\label{i1}
\ssm(\Sigma^0_{n,n+1})=&
\Sc_0+(-\Sc_2)+(2\Sc_3+\Sc_{21})+(-3\Sc_4-3\Sc_{31}-\Sc_{211})+(4\Sc_5+6\Sc_{41}+4\Sc_{311}+\Sc_{2111})\\
\nonumber &+(-5\Sc_6-10\Sc_{51}-10\Sc_{411}+\Sc_{33}-5\Sc_{3111}-\Sc_{21111})+\ldots.
\end{align}
The signs of the term $\Sc_0$, and also of $\Sc_{33}$ violate the pattern (and many more in higher degrees).

Yet, we conjecture that there is a sign pattern. In fact we expect that the following {\em two-step positivity} property holds in general, for quivers, singularities, and
maybe other geometrically relevant varieties:
\begin{itemize}
\item {\em The SSM classes of geometrically interesting varieties, expanded in the new building blocks, the $\tssm_\lambda$ functions of Section \ref{sec:GFpartition},
    have non-negative coefficients.}
\item {\em The coefficients of the Schur expansions of the $\tssm_\lambda$ functions have alternating signs.}
%\item {\em The $\tssm_\lambda$ functions, expanded in $\Sc_\lambda$ functions, have alternating sign coefficients.}
%\item {\em The $\tssm_\lambda$ functions are stabilized versions of the SSM classes of matrix Schubert cells}
\end{itemize}

The second line is our Conjecture \ref{conj:positivity}. We prove the first line in two cases, one is Theorem \ref{thm:genOmega}, and the other is the $A_2$ quiver
representation in Section \ref{sec:A2}. For example, according to Theorem~\ref{thm:main} we have the positive $\tssm_\lambda$ expansion
\begin{equation}\label{i2}
\ssm(\Sigma^0_{n,n+1})=
\tssm_{0}+\tssm_{1}+\tssm_{11}+\tssm_{111}+\ldots,
\end{equation}
and the $\tssm_\lambda$ functions have alternating Schur expansions
\begin{align}\label{i3}
\tssm_0=&
\Sc_0+(-\Sc_1)+(\Sc_2+\Sc_{11})+(-\Sc_3-2\Sc_{21}-\Sc_{111})
 +(\Sc_4+3\Sc_{31}+\Sc_{22}+3\Sc_{211}+\Sc_{1111})
+\ldots\\
\nonumber \tssm_1=&
\Sc_1+(-2\Sc_2-2\Sc_{11})+(3\Sc_3+5\Sc_{21}+3\Sc_{111})
 +(-4\Sc_4-9\Sc_{31}-3\Sc_{22}-9\Sc_{211}-4\Sc_{1111})
+\ldots\\
\nonumber \tssm_{11}=&
\Sc_{11}+(-2\Sc_{21}-3\Sc_{111})
 +(3\Sc_{31}+2\Sc_{22}+7\Sc_{211}+6\Sc_{1111})
+\ldots.
\end{align}
As claimed in (\ref{i2}), adding together the expressions in (\ref{i3}) reproduces (\ref{i1})---but the signs in (\ref{i2}) and (\ref{i3}) show pattern. Further examples
illustrating this two-step positivity phenomenon will appear in \cite{Balazs, job}.

\subsection{Conventions}
We will consider varieties over $\C$, and cohomology with coefficient group $\C$. We distinguish between a ``weakly decreasing sequence of non-negative intergers'' and a
``partition''. The first one has a fixed length, while a paritition does not change by adding a 0 at the end.

\subsection{Acknowledgements}
The first author was partially supported by NKFI 112703 and 112735 as well as ERC Advanced Grant LTDBud; the second author was partially supported by NSF DMS-1200685 and by the Simons Foundation grant 523882. We are
grateful to L. Mihalcea,  T. Ohmoto, A. Szenes, A. Varchenko, and A. Weber for useful discussions on the topics of the paper.

\section{Ohmoto's equivariant CSM/SSM classes}

\subsection{Equivariant CSM/SSM theory} \label{sec:introCSM}
Let the algebraic group $G$ act on the smooth algebraic variety $M$, and let $f:M\to \C$ be an invariant constructible function. In \cite{O1, O2} Ohmoto (see also \cite[Section 2]{W}) defines the
equivariant Chern-Schwartz-MacPherson (CSM) class $\csm(f)\in H_G^*(M)$. The class $\ssm(f)=\csm(f)/c(TM)$ is called the (equivariant) Segre-Schwartz-MacPherson (SSM) class.
In fact the SSM class may have non-zero components in infinitely many degrees, thus it lives in the completion of $H_G^*(M)$. In notation we will not indicate this completion.
If $X$ is an invariant constructible subset of $M$ and $\E_X$ is its indicator function, then one defines $\csm(X \subset M)=\csm(\E_X)$, $\ssm(X \subset M)=\ssm(\E_X)$. If
$M$ is clear from the context we write $\csm(X)$ for $\csm(X\subset M)$ and $\ssm(X)$ for $\ssm(X\subset M)$.

%For a review on the non-equivariant CSM and SSM classes and their role in algebraic geometry see references in \cite{AM1, AM2, O1, O2}.

%We will not recall the definition of these classes, only a list of properties needed below.

Here are some important properties we will use below; for proof see \cite[Section 3.4]{O2}.
\begin{enumerate}[(i)]
\item \label{ismooth}
For $i:X\subset M$, both smooth we have $\csm(X\subset M)=i_*(c(TX))$ and $\ssm(X\subset M)=i_*(c(TX))/c(TM)=i_*(c(-\nu))$ where $\nu$ is the normal bundle of $X \subset
M$.
\item \label{iadd}
For invariant constructible functions $f,g$ and scalars $k,l\in \C$ we have  $\csm(kf+lg)=k\csm(f)+l\csm(g)$, $\ssm(kf+lg)=k\ssm(f)+l\ssm(g)$.
\item \label{ielore}
If $\eta:Y\to M$ is an equivariant map between smooth manifolds then
\[
\eta_* ( c(TY) )= \sum_a a \csm(M_a),
\]
where $M_a=\{m\in M: \chi(\eta^{-1}(m))=a\}$.
\item \label{itrans}
For $Y$ and $M$ smooth let $X\subset M$ be an invariant subvariety with an invariant Whitney stratification. Assume $\eta: Y \to M$ is (equivariant and) transversal to the strata of $X$. Then $\ssm(\eta^{-1}(X))=\eta^*( \ssm(X) )$.
\end{enumerate}

The orbit stratification of an algeraic group action with finitely many orbits is a Whitney stratification, see e.g. the main result of \cite{K} and mathoverflow.net/questions/129218. Below we will apply (\ref{itrans}) to such situations without mentioning the existence of Whitney stratifications.

\smallskip

\begin{remark} \rm
Note that CSM classes behave nicely with respect to direct image (\ref{ielore}), while the closely related SSM classes behave nicely with respect to transversal pullback
(\ref{itrans}).
\end{remark}

\begin{remark} \rm
In other versions of CSM theory the class $\csm(X)$ is an element in the (Borel-Moore) homology of $X$. The cohomology version we use is then obtained by pushing forward the
homology CSM class to the homology of $M$ and applying (equivariant) Poincar\'e duality.
\end{remark}

\subsection{Equivariant CSM/SSM classes in representations}

Let us assume that the underlying space $M$ is a vector space, and rename it to $V$. Then the CSM and SSM classes of $X\subset V$ are in $H^*_G(V)=H^*(BG)$, hence they are
$G$-characteristic classes. The main importance of this special case is the following ``degeneracy locus'' interpretation of SSM classes, which is essentially a consequence
of (\ref{itrans}).

\begin{theorem} \cite[Theorem 3.12 (4)]{O2} \label{thm:ohmoto}
Let $G$ act on $V$ and $X\subset V$ be an invariant subvariety. Let $B$ be smooth and compact and let $E\to B$ be a vector bundle with fiber $V$ and structure group $G$. Let
$X(E)$ be the union of $X$'s in each fiber, and assume that the section $\sigma$ is transversal to $X(E)$. Then the {\em ordinary} (that is, {\em non-equivariant}) SSM-class
of $\sigma^{-1}(X(E))\subset B$ can be obtained as $\ssm(X)$ (as a $G$-charactersitic class) evaluated at the bundle $E\to B$.
\end{theorem}

Here are some other properties we will need below.
\begin{enumerate}[(i)]
\setcounter{enumi}{4}
\item \label{iPart} Let $X_0 \subset X \subset V$ be invariant subvarieties and let $X_0$ be smooth. Assume that there is an invariant complementary subspace $W\leq V$ to
    $T_0X_0$ transversal to $X$. Then $\csm(X)$ is divisible by $c(T_0X_0)$.
\item \label{iEuler}
 Let $X\subset V=\C^n$ be an invariant cone-subvariety (i.e. stable w.r.t. multiplication by $\lambda\in \C$).  We have
\[
\csm(X)=[X]+\ldots+e(V), \qquad \ssm(X)=[X]+\ldots,
\]
that is, the smallest degree part of both is the (equivariant) fundamental class. The class $\csm$ has finitely many non zero components, the highest degree one is the
Euler class $e(V)=\prod_{i=1}^n w_i$ for the weights $w_i$ of the representation. (We choose a maximal torus of $G$, and on weights we will always mean the weights of the
corresponding torus action.)
\item \label{iEuler1}
Let the representation on $V=\C^n$ have $k$ zero weights. That is, the zero weight subspace $V_0 \subset V$ has dimension $k$. Assume that $W$ is an invariant
complementary subspace to $V_0$ and is transversal to the invariant cone-subvariety $X$. Then
\[
\csm(X)=[X]+\ldots+\prod_{i=1}^{n-k} w_i,
\]
where $w_i$ are the non-zero weights of $V$. That is, the highest degree part of $\csm(X)$ has degree $n-k$ and it is the product of the non-zero weights.
\end{enumerate}
\smallskip

In particular, Property (\ref{iEuler}) claims that the $n$th component of $\csm(X)$ is independent of $X$, it only depends on the representation. The essence of
(\ref{iEuler1}) compared to (\ref{iEuler}) is that not only the $n$'th, but the $n-1$'st, $\ldots$, $n-k$'th components of $\csm(X)$ are also independent of $X$.

Property (\ref{iPart}) is a consequence of (\ref{itrans}) as is shown in \cite[Section 6]{RV}. The proof of (\ref{iEuler}) is \cite[Section 4.1]{O1}. Here we
prove~(\ref{iEuler1}). From (\ref{itrans}) we know that $\csm(X)/c(V)=\csm(X\cap W)/c(W)$. Since $c(V)=c(W)$, applying (\ref{iEuler}) to $X\cap W \subset W$ proves
(\ref{iEuler1}).

In concrete examples---e.g. the ones we will deal with in the paper---the existence of $W$ (in Properties (\ref{iPart}) and (\ref{iEuler1})) can easily be checked. In fact
passing to the maximal compact torus (which does not affect equivariant cohomology) it can be proved in very general situations.

\begin{example} \rm
Let the 2-torus act on $\C^3$ by $(a,b).(x,y,z)=(ax, b^2y,z)$. Let $\alpha, \beta$ be the first Chern classes of the  2-torus corresponding to $a,b$. Thus the weights of
this representation are $\alpha, 2\beta$, and 0. Let $X=\{x=0\}$, $Y=\{y=0\}$, $Z=\{z=0\}$. It is instructive to verify Property~(\ref{iEuler1}) in the examples
\[
\csm(X)=(1+2\beta)\alpha, \quad
\csm(Y)=(1+\alpha)2\beta, \quad
%\csm(Z)=0, \qquad
\csm(X \cap Y)=2\alpha\beta,\]
\[
\csm(X \cup Y)=(1+2\beta)\alpha+(1+\alpha)2\beta-2\alpha\beta=\alpha+2\beta+2\alpha\beta.
\]
The claims in the first line follow from (\ref{ismooth}) and the claim in the second line follows from (\ref{iadd}).
\end{example}

\subsection{Interpolation characterization}

Consider the linear representation $V$ of the algebraic group $G$. Assume it has finitely many orbits, and assume that the representation contains the scalars, that is,
orbits are invariant under multiplication by $\lambda\in \C^*$. For an orbit $\Omega$ let $G_\Omega$ be the stabilizer subgroup of a point $x_\Omega\in \Omega$, and let
          \[ \phi_\Omega:H^*(BG)\to H^*(BG_\Omega)\]
 be induced by the inclusion $G_\Omega\subset G$. Let $T_\Omega$ be the tangent space of $\Omega$ at $x_\Omega$, and let $N_\Omega=V/T_\Omega$ be the ``normal'' space. The
 group $G_\Omega$ acts of $T_\Omega$ and $N_\Omega$, hence these representations have an equivariant total Chern class ($c$) and an Euler class ($e$) in $H^*(BG_\Omega)$.

We say that the representation with finitely many orbits satisfies the {\em Euler condition} if $e(N_\Omega)$ is not a 0-divisor in $H^*(BG_\Omega)$ for all $\Omega$.  Let
us recall two topological lemmas.

\begin{lemma}\cite[Theorem 3.7]{FR1} \label{lem:euler}
Let $\Theta_1, \Theta_2, \ldots$ be the list of orbits satisfying $i<j \Rightarrow \Theta_i \not\subset \overline{\Theta}_j$ in a representation satisfying the Euler
condition. Suppose $\omega\in H^*_G(V)$ is 0 when restricted to $\Theta_1\cup\ldots\cup \Theta_s$ and it is 0 restricted to $\Theta_{s+1}$. Then $\omega$ is 0 restricted to
$\Theta_1 \cup \ldots \cup \Theta_{s+1}$.\qed
\end{lemma}

\begin{lemma} \label{lem:gysin}
Let $W$ be an invariant subspace of the $G$-representation $V$, and let $e\in H_G^*(W)=H^*(BG)$ be the equivariant Euler class of the normal bundle of $W \subset V$. If a
class $\omega\in H^n_G(V)=H^*(BG)$ is supported on $W$ (that is, it is 0 restricted to $V-W$), then it is divisible by $e$.
\end{lemma}

\begin{proof} The statement follows from the exactness of the Gysin sequence
\[ H_G^{n- \codim(W\subset V)}(W) \to H^n_G(W) \to H_G^n(V-W)\]
where the first map is multiplication by $e$, and the second map is the composition
$H^n_G(W)=H_G^n(\text{pt})=H^n_G(V) \xrightarrow{r} H_G^n(V-W)$ with $r$ being the restriction map.
\end{proof}

For a cohomology class $x=x_0+x_1+x_2+ \ldots \in H^*(X)$, $x_i\in H^{2i}(X)$, let $\deg(x)$ be the largest $i$ for which $x_i\not=0$. We set $\deg(0)=-\infty$.

\begin{theorem} \label{thm:intchar}
Let the $G$ representation $V$ contain the scalars, let it have finitely many orbits, and let it satisfy the Euler condition. The properties
\begin{enumerate}[(I)]
\item \label{iprinc}
$\phi_\Omega(\csm(\Omega))=c(T_\Omega)e(N_\Omega) \in H^*(BG_\Omega)$,
\item \label{idivis}
$\phi_\Theta(\csm(\Omega))$ is divisible by $c(T_\Theta)$ in $H^*(BG_\Theta)$,
\item \label{idegree}
if $\Theta\not= \Omega$ then $\deg(\phi_\Theta(\csm(\Omega)))< \deg( c(T_\Theta)e(N_\Theta))$
\end{enumerate}
uniquely determine $\csm(\Omega)$.
\end{theorem}

\begin{proof}
First we prove that $\csm(\Omega)$ satisfies the properties.

The orbit $\Omega$ is smooth at $x_\Omega$, hence the image of $\csm(\Omega)$ at the homomorphism
$H^*_G(V) \to H^*_{G_\Omega}(V) = H^*_{G_\Omega}(x_\Omega)$
is $c(T_\Omega)e(N_\Omega)$, see (\ref{ismooth}). The named homomorphism can be identified with $\phi_\Omega$, hence Property (\ref{iprinc}) is proved.

Let
\begin{equation}\label{eqn:function}
\E_\Omega=\sum_{\Phi\leq \Omega} d_{\Omega,\Phi} \E_{\overline{\Phi}},
\end{equation}
where $\Phi\leq \Omega$ means that $\Phi\subset \overline{\Omega}$. Then
$\csm(\Omega)=\sum_{\Phi\leq\Omega} d_{\Omega,\Phi} \csm(\overline{\Phi})$ and
\[
\phi_\Theta(\csm(\Omega))=\sum_{\Theta\leq\Phi\leq\Omega} d_{\Omega,\Phi} \phi_\Theta(\csm(\overline{\Phi})).
 \]
Each of the $\phi_\Theta(\csm(\overline{\Phi}))$ restrictions are divisible by $c(T_\Theta)$, because of (\ref{iPart}). This proves Property~(\ref{idivis}).

Observe that the number of 0 weights of $G_\Theta$ acting on the tangent space of $V$ at $M_\Theta$ is $n-\deg( c(T_\Theta)e(N_\Theta))$. Hence, for any $i\geq \deg(
c(T_\Theta)e(N_\Theta))$ the $i$'th component of $\phi_\Theta(\csm(\overline{\Phi}))$ does not depend on $\Phi$, let the common value be called $x_i$. Then for $i\geq \deg(
c(T_\Theta)e(N_\Theta))$ we have that the $i$'th component of $\phi_\Theta(\csm(\Omega))$ is
\begin{equation}\label{eqn:xd}
x_i\cdot \sum_{\Theta\leq\Phi\leq\Omega} d_{\Omega,\Phi}.
\end{equation}
However, substituting $M_\Theta$ in the identity (\ref{eqn:function}) we get $0=\sum_{\Theta\leq\Phi\leq\Omega} d_{\Omega,\Phi}$. Hence expression (\ref{eqn:xd}) is 0 for
all $i\geq \deg( c(T_\Theta)e(N_\Theta))$, which proves Property (\ref{idegree}).

The proof of the uniqueness of classes satisfying (\ref{iprinc})--(\ref{idegree}) is an adaptation of the argument in \cite[Section 3.3]{MO}. Suppose two classes satisfy the
conditions above for $\csm(\Omega)$, and let $\omega$ be their difference. Then for {\em every} $\Theta$ we have that $\phi_\Theta(\omega)$ is divisible by $c(T_\Theta)$ and
has degree strictly less than $\deg ( c(T_\Theta)e(N_\Theta))$. Let $\Theta_1, \Theta_2, \ldots$ be a (finite) list of orbits satisfying $i<j \Rightarrow \Theta_i
\not\subset \overline{\Theta}_j$. We will prove by induction on $s$ that $\omega$ is 0 when restricted to $\Theta_1\cup \ldots \cup \Theta_s$. For $s=0$ the claim holds.
Suppose we know this statement for $s-1$ and want to prove it for $s$. Because of the induction hypotheses, $\omega$ is supported on $\Theta_{s} \cup \Theta_{s+1} \cup
\ldots$. Hence its $\Theta_s$ restriction must be divisible by $e(N_{\Theta_s})$ (Lemma \ref{lem:gysin}). We also know that it is divisible by $c(T_{\Theta_s})$. These
classes are coprime in $H^*(BG_{\Theta_s})$, therefore we have that $\phi_{\Theta_s} (\omega)$ is divisible by $c(T_{\Theta_s})e(N_{\Theta_s})$. Since its degree is strictly
less than that of $c(T_{\Theta_s})e(N_{\Theta_s})$, we have that $\phi_{\Theta_s}(\omega)=0$. Lemma \ref{lem:euler} implies that $\omega$ restricted to $\Theta_1\cup
\Theta_2\cup\ldots\cup\Theta_{s}$ is also zero.
\end{proof}

\begin{remark}
The property
\begin{enumerate}[(I)]
\setcounter{enumi}{3}
\item \label{ihom}
$\phi_\Theta(\csm(\Omega))=0$ for $\Theta\not\subset \overline{\Omega}$
\end{enumerate}
obviously holds too; it is not listed among the axioms above, because it is forced by them.
\end{remark}

\section{Matrix Schubert cells}\label{sec:SchubertCells} \label{sec:mxsc}

%\subsection{Matrix Schubert cells, representatives, dimensions}\label{sec:SchubertCells}

One of our goals in this paper is to give formulas for the CSM/SSM classes of the orbits of a certain representation. These orbits will be called the matrix Schubert cells.

Let us fix nonnegative integers $k\leq n$. Consider the group $GL_k(\C) \times B^-_n$ acting on $\Hom(\C^k,\C^n)$ by $(A,B).M=BMA^{-1}$. Here $B^-_n$ is the Borel subgroup
of $n\times n$ lower triangular matrixes.
The finitely many orbits of this action are parameterized by $d$-element subsets $J=\{j_1 <\ldots <j_d\} \subset \{1,\ldots,n\}$ with $0\leq d\leq k$.
%The corresponding orbit $\Omega_J$ consists of $n\times k$ matrices such that the rank of the submatrix consisting of the top $r$ rows is  $| J \cap \{1,\ldots,r\}|$.
The corresponding orbit is
\begin{equation}\label{eqn:rank}
\Omega_J=\{ M\text{ is an $n\times k$ matrix}: \text{rk(top $r$ rows of $M$})=| J \cap \{1,\ldots,r\}|\}
\end{equation}
A representative of the orbit $\Omega_J$ is the $n \times k$ matrix $M_J$ whose entries are 0's, except the $(j_u,u)$ entries are 1 ($u=1,\ldots,d$). The orbits will be
called matrix Schubert cells, and their closures are usually called matrix Schubert varieties, see e.g.~\cite{ss,km}.

For $J=\{j_1 <\ldots < j_d\} \subset \{1,\ldots,n\}$ we define a few subsets of the entries of $k\times n$ matrices that will be useful later. Let
\begin{align*}
\A_0&=\{ (v,u)\in \{1,\ldots,n\}\times \{1,\ldots,k\} : u\leq d, v=j_u \},
\\
\A_1&=\{ (v,u)\in \{1,\ldots,n\}\times \{1,\ldots,k\} : u\leq d, v < j_u \},
\\
\A_2&=\{ (v,u)\in \{1,\ldots,n\}\times \{1,\ldots,k\} : u\leq d, v > j_u \},
\\
\A_3&=\{ (v,u)\in \{1,\ldots,n\}\times \{1,\ldots,k\} : u > d \},
\\
\A_4&=\{ (v,u)\in \{1,\ldots,n\}\times \{1,\ldots,k\} : \exists w\leq d\ \ v=j_w, u>w \}.
\end{align*}
The set $\T_J=\A_0 \cup \A_2 \cup \A_4$ represents the directions in $\Hom(\C^k,\C^n)$ that are in the tangent space of $\Omega_J$ at $M_J$, and $\N_J=\{1,\ldots,n\}\times
\{1,\ldots,k\}- \A_0 \cup \A_2 \cup \A_4$ represents the directions normal to $\Omega_J$ at $M_J$. Hence $\dim \Omega_J=|\T_J|$, $\codim (\Omega_J \subset
\Hom(\C^k,\C^n))=|\N_J|$.

\begin{example} \rm \label{ex:23}
For $k=2$, $n=3$ there are 7 orbits, corresponding with the subsets $\{1,2\}$, $\{1,3\}$, $\{2,3\}$, $\{1\}$, $\{2\}$, $\{3\}$, $\{\}$ with representatives ($\bullet$ or $.$
stands for 0)
\[
\begin{pmatrix}
1 & \bullet \\
\bullet & 1 \\
\bullet & \bullet
\end{pmatrix}
\begin{pmatrix}
1 & \bullet \\
\bullet & .  \\
\bullet & 1
\end{pmatrix}
\begin{pmatrix}
. & . \\
1 & \bullet  \\
\bullet & 1
\end{pmatrix}
\begin{pmatrix}
1 & \bullet \\
\bullet & . \\
\bullet & .
\end{pmatrix}
\begin{pmatrix}
. &  .\\
1 & \bullet \\
\bullet & .
\end{pmatrix}
\begin{pmatrix}
. & . \\
 .&  .\\
1 & \bullet
\end{pmatrix}
\begin{pmatrix}
. & . \\
. & . \\
. & .
\end{pmatrix}.
\]
In each matrix a $\bullet$ or a 1 indicate boxes corresponding to directions tangent to $\Omega_J$ and the rest (indicated by .) correspond to normal directions.
\end{example}

%From the description of tangent and normal directions of $\Omega_I$ at $M_I$
%We obtain the following formulas for the total Chern class of the tangent bundle $T\Omega_I$ and for the Euler class of the normal bundle $N\Omega_I$:
%\[
%c( T\Omega_I |_{M_I})=\prod_{ (i,j)\in T_I} (1+ \beta_i-\alpha_j),
%\qquad
%e( N\Omega_I |_{M_I})=\prod_{ (i,j)\in N_I} (\beta_i-\alpha_j).
%\]

Let $\alpha_1,\ldots,\alpha_k$ and $\beta_1,\ldots,\beta_n$ be the Chern roots of the group $GL_k(\C)\times B^-_n$. We have
\begin{equation}\label{eqn:thering}
H^*(B(GL_k(\C)\times B^-_n ))=\C[\alpha_1,\ldots,\alpha_k,\beta_1,\ldots,\beta_n]^{S_k},
\end{equation}
and the weights of the representation $\Hom(\C^k,\C^n)$ defined above are $\beta_v-\alpha_u$ for $v=1,\ldots,n$, $u=1,\ldots,k$. The weight space of $\beta_v-\alpha_u$ is
the line corresponding to the $(v,u)$ entry of $\Hom(\C^k,\C^n)$.

The CSM and SSM classes of matrix Schubert cells $\Omega_J$ are hence elements of the ring (\ref{eqn:thering}) and its completion, respectively. To claim the result about
these classes in Section \ref{sec:csmW} we first need to define some important functions in Section \ref{sec:weight}.

\section{Weight functions} \label{sec:weight}
In this section we define some important polynomials that will be identified with CSM classes of matrix Schubert cells in Section \ref{sec:csmW}.

\subsection{Localization form of weight functions}
Let $k\leq n$ and $I\subset \{1,\ldots,k\}$ where $|I|=d\leq k$  and $I=\{i_1<\ldots<i_d\}$.
% !!k helyett n?!!!   NEM.

\begin{definition}
Let $\aalpha=(\alpha_1,\ldots,\alpha_k)$ and $\bbeta=(\beta_1,\ldots,\beta_n)$ and
\[
U_I(\aalpha,\bbeta)=
\prod_{u=1}^d\prod_{v=i_u+1}^n (1+\beta_v-\alpha_u)
\prod_{u=d+1}^k \prod_{v=1}^n (\beta_v-\alpha_u)
\prod_{u=1}^d\prod_{v=1}^{i_u-1} (\beta_v-\alpha_u)
\prod_{u=1}^d \prod_{v=u+1}^k \frac{1+\alpha_u-\alpha_v}{\alpha_u-\alpha_v}.
\]
A permutation $\sigma\in S_k$ acts on a $k$-tuple by permuting the components. Define the ``weight function''
\[
W_I=W_I(\aalpha;\bbeta)=\frac{1}{(k-d)!}\sum_{\sigma\in S_k}  U_I(\sigma(\aalpha);\bbeta).
\]
\end{definition}

Although we omitted from the notation, the function $W_I$ depends on $k$ and $n$ as well; their stabilization properties will be discussed below. Despite their appearance
the weight functions are polynomials with integer coefficients, in fact of degree $kn-d$.

\begin{remark} \rm
Weight functions were used in \cite{TV} to describe q-hypergeometric solutions of the quantum Knizhnik-Zamolodchikov
equations. They also appeared in joint works \cite{GRTV, RTV, RV} with the second author, as key components in identifying cohomology rings with Bethe algebras. In these
past works weight functions were only defined for $d=k$, the present $d<k$ extension is new.
\end{remark}

\begin{remark}
In \cite{RV} the $d=k$ weight functions are divided by a particular factor. It is shown there that these rational functions are (in a suitable sense) representatives of CSM
classes in some quotient rings that are naturally identified with cohomology rings of {\em compact} spaces. See more on this in Appendix B, Chapter \ref{appB}.
\end{remark}

\begin{example}
For $k=1, n=2$ we have
\[
W_{\{1\}}=1+\beta_2-\alpha_1, \quad
W_{\{2\}}=\beta_1-\alpha_1, \quad
W_{\{\}}=(\beta_1-\alpha_1)(\beta_2-\alpha_1).
\]
For $k=2, n=2$ we have
\[
W_{\{1,2\}}=
\frac{(1+\beta_2-\alpha_1)(\beta_1-\alpha_2)(1+\alpha_1-\alpha_2)}{\alpha_1-\alpha_2}+
\frac{(1+\beta_2-\alpha_2)(\beta_1-\alpha_1)(1+\alpha_2-\alpha_1)}{\alpha_2-\alpha_1} =\]
\[
1+\beta_1+\beta_2+2\beta_1\beta_2-(\alpha_1+\alpha_2)(\beta_1+\beta_2)-\alpha_1-\alpha_2+2\alpha_1\alpha_2.
\]
\end{example}

\subsection{Residue form of weight functions}
The $\beta_i=0$ ($i=1,\ldots,n$) substitution $W_I(\aalpha;0,\ldots,0)$ of the weight function $W_I$ will be denoted by $W_{I,\beta=0}$.

We will use residue formulas for various functions. Namely, for $\zz=z_1,\ldots,z_\mu$ let $\RES_\mu(f(\zz))$ be a short hand notation for
$\Res{z_\mu=\infty}\ldots \Res{z_2=\infty} \Res{z_1=\infty} (f(\zz))$.

\begin{theorem} \label{thm:Wres}
For $d\leq k \leq n$, $|I|=d$ let $r=k-d$ and
\[
f_I=
\prod_{a=1}^r z_a^{n+r-a}
\prod_{a=r+1}^k z_a^{i_{k+1-a}-1}
\prod_{a=r+1}^k (1+z_a)^{n-i_{k+1-a}}
\prod_{a=r+1}^k \prod_{b=1}^{a-1} (1+z_b-z_a)
\prod_{1\leq b<a\leq k} (z_a-z_b).
\]
We have
\[
W_{I,\beta=0}=
(-1)^k
\RES_k
\left(
\frac{f_I}{\prod_{u=1}^k\prod_{v=1}^k (z_u+\alpha_v)} dz_1\ldots dz_k
\right)
\]
\end{theorem}

\begin{proof}
We have
\[
U_{I,\beta=0}=
\prod_{u=1}^d (1-\alpha_u)^{n-i_u}
\prod_{u=d+1}^k (-\alpha_u)^n
\prod_{u=1}^d(-\alpha_u)^{i_u-1}
\prod_{u=1}^d \prod_{v=u+1}^k \frac{1+\alpha_u-\alpha_v}{\alpha_u-\alpha_v}.
\]
Temporarily denote $\alpha_u=\omega_{k+1-u}$, that is consider the list of $\alpha_i$ variables backwards. After rearrangements we obtain
\[
U_{I,\beta=0}=
\prod_{a=1}^r (-\omega_a)^n
\prod_{a=r+1}^k (-\omega_a)^{i_{k+1-a}-1}
\prod_{a=r+1}^k (1-\omega_a)^{n-i_{k+1-a}}
\prod_{a=r+1}^k \prod_{b=1}^{a-1} \frac{1+\omega_a-\omega_b}{\omega_a-\omega_b}.
\]
Define
\begin{equation}\label{eqn:VI}
V_I=U_{I,\beta=0} \cdot \prod_{a=1}^r (-\omega_a)^{r-a} \prod_{a=1}^r\prod_{b=1}^{a-1}\frac{1}{\omega_a-\omega_b}.
\end{equation}
We claim that
\begin{equation}\label{eqn:sym}
\sum_{\sigma\in S_r} V_I(\sigma(\omega_1,\ldots,\omega_r),\omega_{r+1},\ldots,\omega_k)=
U_{I,\beta=0}.
\end{equation}
Indeed, since $U_{I,\beta=0}$ is symmetric in $\omega_1,\ldots,\omega_r$, it can be pulled out of the symmetrization, and the symmetrization of the last two factors of
(\ref{eqn:VI}) is well known to be 1.

Another interpretation of (\ref{eqn:sym}) is that the LHS of (\ref{eqn:sym}) equals
\[
\frac{1}{r!}
\sum_{\sigma\in S_r} U_{I,\beta=0}(\sigma(\omega_1,\ldots,\omega_r),\omega_{r+1},\ldots,\omega_k).
\]
Thus, for the weight function we obtain
\begin{equation}\label{eqn:WV}
W_{I,\beta=0}=\sum_{\sigma\in S_k} V_I(\sigma(\omega_1,\ldots,\omega_k)).
\end{equation}
The main observation is that
\[
V_I=\Res{z_k=-\omega_k}\ldots \Res{z_2=-\omega_2} \Res{z_1=-\omega_1}
\left(
\frac{f_I}{\prod_{u=1}^k\prod_{v=1}^k (z_u+\omega_v)} dz_1\ldots dz_k
\right),
\]
or more generally
\[
V_I(\sigma(\omega_1,\ldots,\omega_k))=\Res{z_k=-\omega_{\sigma(k)}}\ldots \Res{z_2=-\omega_{\sigma(2)}} \Res{z_1=-\omega_{\sigma(1)}}
\left(
\frac{f_I}{\prod_{u=1}^k\prod_{v=1}^k (z_u+\omega_v)} dz_1\ldots dz_k
\right),
\]
and that there are no other non-zero residues of this differential form in $\C^n$. Hence (\ref{eqn:WV}) and iterated applications of the Residue Theorem proves our theorem.
\end{proof}

\section{CSM classes of matrix Schubert cells are weight functions} \label{sec:csmW}

Now we calculate the CSM classes of matrix Schubert cells.

\begin{theorem} \label{thm:csmW}
Consider the $GL_k(\C)\times B^-_n$ representation $\Hom(\C^k,\C^n)$ and the description of the orbits in Section \ref{sec:SchubertCells}. For the  equivariant
Chern-Schwartz-MacPherson class of the orbit $\Omega_I$ we have
\[
\csm(\Omega_I)=W_I(\aalpha,\bbeta).
\]
\end{theorem}

% !! ebből kijon a kempf-laksov?!!

\begin{proof}
We will show that $W_I$ satisfies the properties of Theorem \ref{thm:intchar} for the representation $\Hom(\C^k,\C^n)$.

Let $J=\{j_1<\ldots<j_d\}$, $d\leq k$. By looking at the matrix $M_J$ one finds that the maximal torus of $G_{\Omega_J}$ is of rank $n+k-d$ and the map $\phi_J:H^*_{GL_k(\C)
\times B_n^-}(V) \to H^*(BG_{\Omega_J})$ can be described by
\[
\alpha_u\mapsto
\begin{cases}
\beta_{j_u} & u=1,\ldots,d \\
\alpha_u & u=d+1,\ldots, k,
\end{cases}
\qquad\qquad
\beta_v\mapsto \beta_v, v=1,\ldots,n.
\]
Using the notations in Section \ref{sec:SchubertCells} we have that %!!zarojeleket nagyitottam!!
\[
c(T_{\Omega_j})=
\phi_J\left( \prod_{(v,u)\in \T_J} (1+\beta_v-\alpha_u)\right),
\
e(N_{\Omega_J}) =
\phi_J\left(  \prod_{(v,u)\in \N_J} (\beta_v-\alpha_u)\right),
\]
and that $\deg\big( c(T_{\Omega_J}) e(N_{\Omega_J})\big)=nk-d$. It follows that the representation satisfies the Euler condition.

Let $I=\{i_1<\ldots<i_e\}$ and recall the definition of $U_I$:
\[
U_I=
\underbrace{\prod_{u=1}^e\prod_{v=i_u+1}^n (1+\beta_v-\alpha_u)}_{P_1}
\underbrace{\prod_{u=e+1}^k \prod_{v=1}^n (\beta_v-\alpha_u)}_{P_2}
\underbrace{\prod_{u=1}^e\prod_{v=1}^{i_u-1} (\beta_v-\alpha_u)}_{P_3}
\underbrace{\prod_{u=1}^e \prod_{v=u+1}^k \frac{1+\alpha_u-\alpha_v}{\alpha_u-\alpha_v}}_{P_4}.
\]

We have that $\phi_J(W_I)=$
\begin{equation}\label{eqn:phiU}
\phi_J\left( \frac{1}{(k-e)!} \sum_{\sigma\in S_k} U_I( \sigma(\aalpha);\bbeta)\right)
=
\frac{1}{(k-e)!} \sum_{\sigma\in S_k} U_I( \sigma( \beta_1,\ldots,\beta_d,\alpha_{d+1},\ldots,\alpha_k);\bbeta).
\end{equation}
The main observation of the proof is that, due to factors of $P_2$ and $P_3$ we have
\begin{equation}\label{obs}
e<d
 \text{ or } \exists a\in\{1,\ldots,d\} \text{ s.t. } i_{\sigma(u)}>j_u \qquad \Rightarrow \qquad
U_I( \sigma( \beta_1,\ldots,\beta_d,\alpha_{d+1},\ldots,\alpha_k);\bbeta)=0.
\end{equation}
Therefore in the rest of the proof we will assume that $e\geq d$, and define $S_k^*$ by $\sigma\in S_k^*$ if $i_{\sigma(u)}\leq j_u$ for all $u=1,\ldots,d$. Then we have
\begin{equation} \label{eqn:Skstar}
\phi_J(W_I)=\frac{1}{(k-e)!} \sum_{\sigma\in S_k^*} \phi_J( U_I(\sigma(\aalpha);\bbeta)).
\end{equation}

Now we are ready to prove Properties (\ref{iprinc})-(\ref{idegree}).

If $I=J$ (in particular $d=e$) then $\sigma\in S_k^*$ iff $\sigma(u)=u$ for $u=1,\ldots,e$. Hence there are $(k-e)!$ terms in (\ref{eqn:Skstar}) and each of them is
\[\phi_J\left(
\prod_{(v,u)\in \A_2} (1+\beta_v-\alpha_u)
\prod_{(v,u)\in \A_3} (\beta_v-\alpha_u)
\prod_{(v,u)\in \A_1} (\beta_v-\alpha_u)
\prod_{(v,u)\in \A_4} \frac{( 1+\beta_v-\alpha_u)}{(\beta_v-\alpha_u)}
\right)=
\]
\[
\phi_J\left(
\prod_{(v,u)\in \T_I} (1+\beta_v-\alpha_u)
\prod_{(v,u)\in \N_I} (\beta_v-\alpha_u)
\right)=c(T_{\Omega_J})e(N_{\Omega_J}).
\]
This proves Property (\ref{iprinc}).

To prove Property (\ref{idivis}) we need to show that $\prod_{(v,u)\in \A_0 \cup \A_2 \cup \A_4} \phi_J(1+\beta_v-\alpha_u)$ divides the expression in~(\ref{eqn:Skstar}). We
claim that this divisibility holds for every term of~(\ref{eqn:Skstar}). A term of (\ref{eqn:Skstar}) is a product of $\phi_J$-images of the factors in $P_1, P_2, P_3$, and
$P_4$. For $(v,u)\in \A_0$ we have $\phi_J(1+\beta_v-\alpha_u)=1$. For $(v,u)\in \A_2$ the factor $\phi_J(1+\beta_v-\alpha_u)$ appears as one factor in $\phi_J(P_1)$
(because of $\sigma\in S_k^*$). If $(v,u)\in \A_4$ then the factor $\phi_J(1+\beta_v-\alpha_u)$ appears either as a factor of $\phi_J(P_1)$  or $\phi_J(P_4)$ (again, because
of $\sigma\in S_k^*$). The factors of $\prod_{(v,u)\in \A_2 \cup \A_4} \phi_J(1+\beta_v-\alpha_u)$ are all different, hence we proved the divisibility Property
(\ref{idivis}).

To prove Property (\ref{idegree}) recall that if $e<d$ then $\phi_J(W_I)=0$. If $e>d$ then
\[
\deg(\phi_J(W_I))\leq  \deg(W_I) =nk-e < nk-d=\deg ( c(T_{\Omega_J})e(N_{\Omega_J})).
\]
Let us assume that $d=e$ but $J\not= I$. Then in each term of (\ref{eqn:Skstar}) there is an $u\in \{1,\ldots,d\}$ for which $j_u>i_{\sigma(u)}$. This implies that among the
factors of $\phi_J(P_1)$ one of them is $\phi_J(1+\beta_{j_u}-\beta_{j_u})=1$. Hence
\[
\deg(\phi_J(W_I))\leq  \deg(W_I)-1 =nk-e -1 < nk-d=\deg ( c(T_{\Omega_J})e(N_{\Omega_J})),
\]
which completes the proof.
\end{proof}

\begin{corollary}\label{cor:osztva}
Consider the $GL_k(\C)\times B^-_n$ representation $\Hom(\C^k,\C^n)$ and the description of the orbits in Section \ref{sec:SchubertCells}. For the  equivariant
Segre-Schwartz-MacPherson class of the orbit $\Omega_I$ we have
\[
\ssm(\Omega_I)=\frac{W_I(\aalpha,\bbeta)}{\prod_{u=1}^k\prod_{v=1}^n (1+\beta_v-\alpha_u)}.
\]
\qed
\end{corollary}

\noindent The $GL_k(\C)$-equivariant CSM and SSM classes of $\Omega_I$ are hence $W_{I,\beta=0}$ and $W_{I,\beta=0}/\prod_{u=1}^k (1-\alpha_u)^n$.

\section{Symmetric functions. Residue generator functions. } \label{sec:sym}

In Sections \ref{sec:gen} and \ref{sec:GFpartition} we will give generating function descriptions of certain CSM/SSM classes. In this section we recall Schur functions, and
develop the (``iterated residue'') generating function tool we will use later.

Below we will work with integer vectors $(\lambda_1,\ldots,\lambda_\mu)$, and some of them will be weakly decreasing, i.e. satisfying $\lambda_i\geq \lambda_{i+1}$. A
partition is a class of weakly decreasing integer vectors generated by the relation $(\lambda_1,\ldots,\lambda_\mu)\sim(\lambda_1,\ldots,\lambda_\mu,0)$.

Let us warn the reader that certain theorems will deal with weakly decreasing integer vectors, and in those for example $(3,1)$ and $(3,1,0)$ are {\em different} integer
vectors.

\subsection{Schur functions}

Let $c_i$, $i=1,2,\ldots$, be a sequence of variables and set $c_{<0}=0$, $c_0=1$, and declare $\deg(c_i)=i$. For an integer vector
$\lambda=(\lambda_1,\ldots,\lambda_\mu)\in \Z^\mu$ define
\[
\Sc_\lambda=\det( c_{\lambda_i+j-i} )_{i,j=1,\ldots,\mu} \in \C[c_1,c_2,\ldots].
\]
If $\Sc_\lambda\not=0$ then its degree is $|\lambda|=\sum \lambda_i$. We have $\Sc_\lambda=\Sc_{\lambda,0}$ as well as the straightening laws
\begin{equation}\label{eqn:str}
\Sc_{I,a,b,J} = -\Sc_{I,b-1,a+1,J},
 \quad \Sc_{I,a,a+1,J}=0.
\end{equation}
The collection of $\Sc_\lambda$'s for partition $\lambda$'s is a basis of the vector space of polynomials in $c_i$. For $\lambda$ a partition $\Sc_\lambda$ is called a Schur
function, other $\Sc_\lambda$ will be called fake Schur functions.
Later we will also deal with formal (infinite) sums of $\Sc_\lambda$'s, i.e. we formally work in the completion $\C[[c_1,c_2,\ldots]]$. Since the straightening laws respect
degree, an infinite sum of $\Sc_\lambda$'s make sense as long as for every $n$ there are finitely many terms for which $|\lambda|=n$.

Certain substitutions will play a key role below. Namely, let $\alpha_1,\ldots,\alpha_k$, and $\beta_1,\ldots,\beta_n$ be two finite sets of variables (all declared to have
degree 1), and define
\[
\rho^{k,n}:
\C[c_1,c_2,\ldots] \to
\
\C[\alpha_1,\ldots,\alpha_k;\beta_1,\ldots,\beta_n]^{S_k \times S_n}
\]
by
\begin{equation} \label{eqn:defc}
\rho^{k,n}(1+c_1t+c_2t^2+\ldots)= \frac{\prod_{i=1}^n (1+\beta_it)}{\prod_{i=1}^k (1+\alpha_it)}.
\end{equation}
%That is, $\rho^{k,n}(\Sc_\lambda)$ is obtained from $\Sc_{\lambda)$ by substituting the Taylor coefficients of the right hand side of (\ref{eqn:defc})
%In longer formulas we will just write $\Sc_\lambda$ for $\Sc_\lambda(\aalpha,\bbeta)$ and will explicitly state how many $\alpha_i$ and $\beta_j$ variables are meant.
%!! javaslom a $\rho^{k,n}$ jeloles bevezeset, szoval pl $\Sc^{k,n}_\lambda=\rho^{k,n}\Sc_\lambda$. kesobb nehez kovetnem mi "ugyanaz"!!

For example $\rho^{k,n}(\Sc_1)=\sum_{j=1}^{n}\beta_j-\sum_{i=1}^k \alpha_i$, and
\[
%\rho^{k,n}(\Sc_1)=\sum_{j=1}^{n}\beta_j-\sum_{i=1}^k \alpha_i, \qquad
\rho^{k,0}(\Sc_{11})=\sum \{\alpha_i\alpha_j : 1\leq i<j\leq k\},
\qquad
%\Sc^{0,n}_1=\sum_{i=1}^n \beta_i, \qquad
\Sc^{0,n}_{11}=\sum \{\beta_i\beta_j : 1\leq i \leq j\leq n \}.
\]

\begin{lemma} \label{lem:schur}
For an integer vector $\lambda=(\lambda_1,\ldots,\lambda_\mu)$ we have
\[
\Sc_\lambda=(-1)^\mu
\RES_\mu
\left(
\prod_{i=1}^{\mu} z_i^{\lambda_i}
\cdot
\prod_{1\leq i < j \leq \mu} \left(1-\frac{z_i}{z_j}\right)
\cdot
\prod_{i=1}^{\mu}
\sum_{u=0}^\infty \frac{c_u}{z_i^u}
\cdot
\prod_{i=1}^\mu \frac{dz_i}{z_i}
\right)
\]
\[
\rho^{k,n}(\Sc_\lambda)=(-1)^\mu
\RES_\mu
\left(
\prod_{i=1}^{\mu} z_i^{\lambda_i}
\cdot
\prod_{1\leq i < j \leq \mu} \left(1-\frac{z_i}{z_j}\right)
\cdot
\prod_{i=1}^{\mu}
\frac{  \prod_{u=1}^n (1+\beta_u/z_i) }{ \prod_{u=1}^k (1+\alpha_u/z_i)}
\cdot
\prod_{i=1}^\mu \frac{dz_i}{z_i}
\right)
\]
\end{lemma}

\begin{proof}
Residue formulas for Schur functions---in different disguise---are well known, hence here we only sketch the proof. Write the defining determinant of $\Sc_\lambda$ or
$\rho^{k,n}(\Sc_\lambda)$ as a sum of terms corresponding to permutations. Each term of this sum is the residue of the differential form named in the theorem at a point of
$\C^\mu$. The form has no other finite residues in $\C^\mu$, hence the sum of these residues is $(-1)^\mu$ times the residues at infinity (cf. proof of Theorem
\ref{thm:Wres}).
\end{proof}

\subsection{The $\SSS$ operation}

Let $z_1,\ldots,z_\mu$ be an ordered set of variables. For a monomial $z_1^{\lambda_1}\ldots z_\mu^{\lambda_\mu}$ define
\[
\SSS_{z_1,\ldots,z_\mu}( z_1^{\lambda_1}\ldots z_\mu^{\lambda_\mu} ) = \Sc_{\lambda_1,\ldots,\lambda_\mu}.
\]
For polynomials in $z_1,\ldots,z_\mu$ we extend this operation linearly.
%For example
%\[
%\SSS_{z_1,z_2}(z_1^2z_2^2-5z_1z_2^3+z_1^{-1}z_2^3)=\Delta_{2,2}-5\Delta_{1,3}-\Delta_{-1,3}=6\Delta_{2,2}+\Delta_{2,0}.
%\]
Since the {\em straightening rules} (\ref{eqn:str}) respect the degree the operation extends formally to formal power series, resulting in infinite sums of $\Sc_\lambda$'s,
that is, formal power series in $c_i$'s. For example
\begin{equation} \label{eqn:Sex}
\SSS_{z_1,z_2}\left(z_1^3z_2^1 \cdot \sum_{i=0}^{\infty} z_2^i\right)=
\sum_{i=1}^{\infty} \Sc_{3i}=
\Sc_{31}+\Sc_{32}+\Sc_{33}-
\sum_{i=4}^{\infty} \Sc_{i4}.
\end{equation}
Observe that the middle expression is an expansion in terms of (possibly) fake Scher functions, and the last expression is an expansion in terms of Schur functions.

We will use {\em certain} rational functions to encode formal power series. Namely, by convention, the rational functions of the form
\[
\frac{p(z_1,\ldots,z_\mu)}{ \prod_{i=1}^\mu  (1+\kappa_i z_i)^{l_i}},
\]
where $p$ is a polynomial, $\kappa_i\in \Z$, will denote the formal power series obtained by replacing each $1/(1+\kappa_i z_i)$ factor by $\sum_{j=0}^{\infty} (-\kappa_i
z_i)^j$. For example, by $\SSS_{z_1,z_2}(z_1^3z_2/(1-z_2))$ we mean the same expression as (\ref{eqn:Sex}).

%Recall that (\ref{eqn:defc}) defines certain substitutions of $\aalpha, \bbeta$-polynomials into $c_i$. By $\SSS_{\zz}^{k,n}(f(\zz))$ we mean the substitution of these
%values into $\SSS_{\zz}(f(\zz))$. Thus, $\SSS_{\zz}^{k,n}(f(\zz))$ is a formal power series in
%$\alpha_1,\ldots,\alpha_k, \beta_1,\ldots,\beta_n$.

Define $\SSS_{\zz}^{k,n}(f(\zz))=\rho^{k,n}(\SSS_{\zz}(f(\zz)))$.

The following proposition---which follows directly from Lemma \ref{lem:schur}---is the reason for calling the $\SSS$-operation the ``iterated residue operation''.

\begin{proposition} \label{prop:RESS}
For a polynomial or formal power series $p(z_1,\ldots,z_\mu)$ we have
\begin{equation}\label{eqn:SchurRES1}
(-1)^\mu
\RES_\mu
\left(
%\prod_{i=1}^{\mu} z_i^{\mu_i}
p(\zz)
\cdot
\prod_{1\leq i < j \leq \mu} \left(1-\frac{z_i}{z_j}\right)
\cdot
\prod_{i=1}^{\mu}
\sum_{u=0}^\infty \frac{c_u}{z_i^u}
\cdot
\prod_{i=1}^\mu \frac{dz_i}{z_i}
\right)=
\SSS_{z_1,\ldots,z_\mu}(p(\zz)),
\end{equation}
\begin{equation}\label{eqn:SchurRES2}
(-1)^\mu
\RES_\mu
\left(
%\prod_{i=1}^{\mu} z_i^{\mu_i}
p(\zz)
\cdot
\prod_{1\leq i < j \leq \mu} \left(1-\frac{z_i}{z_j}\right)
\cdot
\prod_{i=1}^{\mu}
\frac{  \prod_{u=1}^n (1+\beta_u/z_i) }{ \prod_{u=1}^k (1+\alpha_u/z_i)}
\cdot
\prod_{i=1}^\mu \frac{dz_i}{z_i}
\right)=
\SSS^{k,n}_{z_1,\ldots,z_\mu}(p(\zz)).
\end{equation}
\qed
\end{proposition}

\section{Generating functions parameterized by weakly decreasing sequences}\label{sec:gen}

In this section we prove a generating sequence descriptions of the $GL_k(\C)$-equivariant CSM and SSM classes of matrix Schubert varieties, namely Theorem \ref{thm:genWD}
and Corollary \ref{cor:ssmS}. These generating functions will depend on weakly decreasing integer sequences. In Section \ref{sec:GFpartition} these results will be improved
to generating sequences depending on partitions.

%{\bf !!az alabbi paragrafus igy szerintem nem ertelmes, ezt beszeljuk meg, mert itt van a fo gondom. Persze megertettem mire gondolsz, de...!!}

%Recall that {\em as a weakly decreasing integer sequence} $(\lambda_1,\ldots,\lambda_\mu)\not=(\lambda_1,\ldots,\lambda_\mu,0)$. It is known that the equivariant
%fundamental class of (the closure of) a matrix Schubert cells does not change when one attaches a 0 to the end of the weakly decreasing integer sequence \cite{ss, km}. We
%will see below that the higher order terms of CSM and SSM classes {\em do change} with this operation.

%To parameterize (matrix) Schubert cells one may use either subsets $I\subset \{1,\ldots,n\}$ as in Section~\ref{sec:SchubertCells}, or weakly decreasing integer sequences.
It is a remarkable fact of Schubert calculus, that the equivariant fundamental class of (the closure of) a matrix Schubert cells does not change when one attaches a 0 to the
end of the weakly decreasing integer sequence \cite{ss, km}. We will see below that the higher order terms of CSM and SSM classes {\em change} with this operation. Yet,
there is one version that will depend only on a partition (see Theorem \ref{thm:partition_gen_fv} below).

\subsection{Conventions on integer sequences.}
\label{sec:conventions}

Recall that the set
\[
\I_{k,n}=\{ I :  I=\{i_1<\ldots<i_d\}\subset \{1,\ldots, n\} , 0\leq d\leq k \}
\]
parameterizes the matrix Schubert cells of $\Hom(\C^k,\C^n)$. To an element $I\in \I_{k,n}$ we associate a weakly decreasing sequence $\lambda=(\lambda_1 \geq  \lambda_2
\geq \ldots \geq \lambda_k)$ of non-negative integers by the {\em conversion formula}
\[
\lambda_a=i_{k+1-a}-(k+1-a)\]
for $a=1,\ldots,k$, where by convention $i_{d+a}=n+a$ for $a=1,\ldots,k-d$.

For a sequence $\lambda=(\lambda_1 \geq  \lambda_2 \geq \ldots \geq \lambda_k)$ of weakly decreasing integers, let $I_\lambda=\{i_1< i_2<\ldots< i_k\}\subset \Z$ be defined
by the conversion formula (equivalent to the one above)
\[ i_a=\lambda_{k+1-a}+a. \]
We say that $\lambda=(\lambda_1 \geq  \lambda_2 \geq \ldots \geq \lambda_k)$ and the non-negative integer $n$ are {\em compatible} if the elements in $I_\lambda$ larger than
$n$ form an interval (possibly empty) starting at $n+1$. That is, if there exists a $q\geq 0$ such that $I_\lambda\cap \Z^{>n}=\{n+1,n+2,\ldots,n+q\}$.

It follows that map $I\mapsto \lambda$ described above is a bijection between $\I_{k,n}$ and
\[
\I'_{k,n}=\{ \lambda :  \lambda=(\lambda_1 \geq \ldots \geq \lambda_k)\in \N^k, \lambda \text{ and $n$ are compatible}\}.
\]
The inverse map $\I'_{k,n}\to \I_{k,n}$ is $\lambda \mapsto I_\lambda \cap \{1,\ldots,n\}$.
Observe that a given $\lambda$ is compatible with any sufficiently large $n$.

\begin{example} \rm
Let $k=2, n=3$ and consider the subsets $I\subset \{1,2,3\}$ as in Example \ref{ex:23}. The corresponding $\lambda$'s are $(0,0)$, $(1,0)$, $(1,1)$, $(2,0)$, $(2,1)$,
$(2,2)$, $(3,3)$, respectively.
\end{example}

\begin{example} \rm \label{ex:31}
The sequence $\lambda=(3,1)$ is compatible with $n$ if and only if $n\geq 4$ (since $I_\lambda=\{2,5\}$).
The element corresponding to $\lambda=(3,1)$ in $\I_{2,4}$ is $I=\{2\}$. For $n\geq 5$ the element corresponding to $\lambda=(3,1)$ in $\I_{2,n}$ is $I=\{2,5\}$.
\end{example}

\subsection{Generating functions for $GL_k(\C)$-equivariant CSM and SSM classes}

The $GL_k(\C)\times B^-_n$-equivariant CSM/SSM classes we study are elements of
\[
\C[\alpha_1,\ldots,\alpha_k,\beta_1,\ldots,\beta_n]^{S_k},\qquad \C[[\alpha_1,\ldots,\alpha_k,\beta_1,\ldots,\beta_n]]^{S_k}.
\]
By plugging in $\beta_i=0$ for all $i=1,\ldots,n$ we obtain symmetric polynomials (power series) in $\alpha_1,\ldots,\alpha_k$, hence linear combinations (formal infinite sums) of polynomials $\rho^{k,0}(\Sc_\lambda)$. The topological counterpart of this substitution is considering equivariant cohomology only with respect to the $GL_k(\C)$
factor of $GL_k(\C)\times B^-_n$.

Denote
\[
\csm_{\beta=0} (\Omega_I)=\csm(\Omega_I)|_{\beta_v=0,v=1,\ldots,n},
\qquad
\ssm_{\beta=0} (\Omega_I)=\ssm(\Omega_I)|_{\beta_v=0,v=1,\ldots,n}.
\]
Our next goal is to find expressions for the Schur function expansions of these functions.

\begin{proposition} \label{prop:csmRES}
For $\lambda \in \I'_{k,n}$ let $I=I_\lambda\cap \{1,\ldots,n\}$ be the corresponding element in $\I_{k,n}$. Let
\[
f_{\lambda,n}=
\prod_{i=1}^k z_i^{\lambda_i+k+1-i}
\prod_{i=1}^k (1+z_i)^{\max(0,n-k-1-\lambda_i+i)}
\mathop{\prod_{1\leq i<j \leq k}}_{\lambda_j-j\leq n-k-1} (1+z_i-z_j).
\]
We have
\[
\csm_{\beta=0}(\Omega_{I_\lambda \cap \{1,\ldots,n\}})=
(-1)^k
\RES_k
\left(
\frac{f_{\lambda,n}}
{\prod_{i=1}^k\prod_{u=1}^k (z_i+\alpha_u)} dz_1\ldots dz_k
\right)
\]
\end{proposition}

\begin{proof}
We have $\csm_{\beta=0}(\Omega_I)=W_{I,\beta=0}$ due to Theorem \ref{thm:csmW}. For the latter we have a residue description, Theorem \ref{thm:Wres}, which is reformulated
here for $\lambda$ instead of $I$.
\end{proof}

\begin{definition} Let $k,n\in \N$. For $\lambda=(\lambda_1\geq \lambda_2 \geq \ldots \geq \lambda_k)$, $\lambda_k\geq 0$ define
\[
\Fcsm_{\lambda,n}= \prod_{i=1}^k z_i^{\lambda_i} \cdot \prod_{i=1}^k (1+z_i)^{\max(0,n-k-1-\lambda_i+i)}  \cdot
\mathop{\prod_{1\leq i<j \leq k}}_{\lambda_j-j\leq n-k-1} (1+z_i-z_j),
\]
\[
\Fssm_{\lambda,n}=\frac{\Fcsm_{\lambda,n}}{\prod_{i=1}^k (1+z_i)^n}.
\]
\end{definition}

Observe that if $n$ is large (in fact $n\geq \lambda_1+k$ \footnote{since $\lambda_1+k=i_k$ in the language of Section \ref{sec:conventions} the condition is $i_k\leq n$})
then $\Fssm_{\lambda,n}$ does not depend on $n$. The stabilized value will be called
\begin{align}
\Fssm_{\lambda,\infty}&=
\prod_{i=1}^k z_i^{\lambda_i} \cdot \prod_{i=1}^k (1+z_i)^{-k-1-\lambda_i+i}  \cdot
\prod_{1\leq i<j \leq k}(1+z_i-z_j)
\label{eqn:Fssm1}
\\
&=\prod_{i=1}^k \left( \frac{z_i}{1+z_i} \right)^{\lambda_i} \prod_{j=1}^k \prod_{i=1}^j \frac{ 1+z_i-z_j}{1+z_i}
\label{eqn:Fssm2}.
\end{align}

\begin{example} We have
\[
\Fssm_{(1,0),0}=z_1, \qquad\qquad\qquad
\Fssm_{(1,0),1}=\frac{z_1(1+z_1-z_2)}{(1+z_1)(1+z_2)},
\]
\[
\Fssm_{(1,0),2}=\frac{z_1(1+z_1-z_2)}{(1+z_1)(1+z_2)^2}, \quad
\Fssm_{(1,0),\infty}=\Fssm_{(1,0),\geq 3}=\frac{z_1(1+z_1-z_2)}{(1+z_1)(1+z_2)^3}.
\]
\end{example}

The following theorem gives the generating sequences of $GL_k(\C)$-equivariant CSM and SSM classes of matrix Schubert cells in $\Hom(\C^k,\C^n)$.

\begin{theorem}\label{thm:genWD}
%Let $\lambda=(\lambda_1 \geq \ldots \geq \lambda_k)$ be a weakly decreasing sequence of non-negative integers, and let $n$ be a compatible integer. Then
For $\lambda \in \I'_{k,n}$ let $I=I_\lambda\cap \{1,\ldots,n\}$ be the corresponding element in $\I_{k,n}$.
Then
\[
\csm_{\beta=0}(\Omega_{I}) = \SSS^{k,0}_{z_1,\ldots,z_k} (\Fcsm_{\lambda,n}),
\quad
\ssm_{\beta=0}(\Omega_{I}) = \SSS^{k,0}_{z_1,\ldots,z_k} (\Fssm_{\lambda,n}).
\]
\end{theorem}

\begin{proof}
The first statement follows from Proposition \ref{prop:csmRES} and Proposition \ref{prop:RESS}. The second statement follows from the first one.
\end{proof}

\begin{example} \rm \label{ex:31v}
Let $k=2$, $\lambda=(3,1)$. Then $I_\lambda=\{2,5\}$. Hence $\lambda$ is compatible with $n$ iff $n\geq 4$. The corresponding subset in $\I_{2,4}$ is $I=\{2\}$, and for
$n\geq 5$ the corresponding subset in $\I_{2,n}$ is $\{2,5\}$. Calculating Taylor series of the appropriate explicit rational functions we obtain that for $n=4$ we have
\begin{align*}
\csm_{\beta=0}(\Omega_{ \{2\} }) &= \SSS^{2,0}_{z_1,z_2}
\left(z_1^3z_2+ ( z_1^4z_2+z_1^3z_2^2)+ (2 z_1^4z_2^2-z_1^3z_2^3)+
(z_1^4z_2^3-z_1^3z_2^4)\right) \\
\ &=
\rho^{2,0}\left(
\Sc_{3,1}+(\Sc_{4,1}+\Sc_{3,2})+(2\Sc_{4,2}-\Sc_{3,3})+\Sc_{4,3}\right).
\\
\ssm_{\beta=0}(\Omega_{ \{2\} }) & = \SSS^{2,0}_{z_1,z_2}
(z_1^3z_2 - (3 z_1^4z_2+3z_1^3z_2^2
)+ (6z_1^5z_2+10 z_1^4z_2^2+5z_1^3z_2^3)
\\
\ & \qquad\qquad\qquad\qquad  - (10z_1^6z_2+22z_1^5z_2^2+17z_1^4z_2^3+7z_1^3z_2^4)+\ldots)
\\
\ &= \rho^{2,0} \left(
\Sc_{3,1}-(3\Sc_{4,1}+3\Sc_{3,2})+(6\Sc_{5,1}+10\Sc_{4,2}+5\Sc_{3,3})-
(10\Sc_{6,1}+22\Sc_{5,2}+17\Sc_{4,3})+\ldots\right).
\end{align*}
For %$\lambda=(3,1)$ and
$n\geq 5$ we have
\begin{align*}
\csm_{\beta=0}(\Omega_{ \{2,5\} }) & = \SSS^{2,0}_{z_1,z_2}
\left(z_1^3z_2+ ( z_1^4z_2+2z_1^3z_2^2)+ (3 z_1^4z_2^2)+
(3z_1^4z_2^3-2z_1^3z_2^4)+(z_1^4z_2^4-z_1^3z_2^5)
\right)
\\
&= \rho^{2,0}\left(
\Sc_{3,1}+(\Sc_{4,1}+2\Sc_{3,2})+(3\Sc_{4,2})+3\Sc_{4,3}+2\Sc_{4,4}\right).
\\
\ssm_{\beta=0}(\Omega_{ \{2,5\} }) & = \SSS^{2,0}_{z_1,z_2}
(z_1^3z_2-(4 z_1^4z_2+3z_1^3z_2^2)+ (13 z_1^4z_2^2+5z_1^3z_2^3+10z_1^5z_2)
\\
& \qquad\qquad\qquad\qquad - (20z_1^6z_2+35z_1^5z_2^2+22z_1^4z_2^3+7z_1^3z_2^4)+\ldots)
\\
&= \rho^{2,0}\left(
\Sc_{3,1}-(4\Sc_{4,1}+3\Sc_{3,2})+(13\Sc_{4,2}+5\Sc_{3,3}+10\Sc_{5,1})-(20\Sc_{6,1}
+35\Sc_{5,2}+22\Sc_{4,3})+\ldots\right).
\end{align*}
%If $n\geq 5$ then $\Fssm_{(3,1),n}$ does not depend on $n$ any longer. Hence the last formula for $\ssm_{\beta=0}(\Omega_{ \{2,5\} })$ holds for any $n\geq 5$.
\end{example}

For $\lambda=(\lambda_1\geq \ldots \geq \lambda_k)$ assume that $n$ is large enough to ensure $\lambda_1\leq n-k$. Then the set in $\I_{k,n}$ corresponding to $\lambda$ is
$I_\lambda$. Also, the elements of the matrix Schubert cell $\Omega_I \subset \Hom(\C^k,\C^n)$ have full rank (i.e. rank $k$).

\begin{corollary}\label{cor:ssmS}
If $\lambda_1\leq n-k$ then
%For the matrix Schubert cell $\Omega_{I_\lambda}\subset \Hom(\C^k,\C^n)$ we have
\begin{equation}\label{eqn:majdnem}
\ssm_{\beta=0}(\Omega_{I_\lambda}) = \SSS^{k,0}_{z_1,\ldots,z_k} (\Fssm_{\lambda,\infty}).
\end{equation} \qed
\end{corollary}
% !!ez igy formalisan nem stimmel, nincs n fugges. talan egy $\cap n$ az indexbol missing?!! De stimmel, az n-fuggest most kiirtam reszletesen

The essence of Corollary \ref{cor:ssmS} is that given a weakly decreasing sequence of non-negative integers $\lambda$, there is a formula (namely the right hand side of
(\ref{eqn:majdnem})) which expresses the $GL_k(\C)$-equivariant SSM class of $\Omega_{I_\lambda} \subset \Hom(\C^k,\C^n)$ {\em for all sufficiently large $n$}.

Unfortunately the expression given in Corollary \ref{cor:ssmS} {\em does} depend on $k$, that is, it changes if we add a 0 to the end of $\lambda$. This will be improved in
Section \ref{sec:GFpartition}.

%\subsection{$\lambda^c$ stability of CSM classes of matrix Schubert varieties}
%A consequence of Theorem \ref{thm:genWD} is a surprising stability of CSM classes of matrix Schubert varieties.

\begin{remark}
Corollary \ref{cor:ssmS} shows the stabilzation of SSM classes when $n\geq k+\lambda_1$. There is another type of stabilization of CSM classes in the $n\geq k+\lambda_1$ range. Namely, in this case 
\[
\Fcsm_{\lambda,n}= \prod_{i=1}^k z_i^{\lambda_i} \cdot \prod_{i=1}^k (1+z_i)^{n-k-1-\lambda_i+i}  \cdot
\prod_{1\leq i<j \leq k} (1+z_i-z_j).
\]
Hence, if $\lambda=(\lambda_1,\ldots,\lambda_k)$ is changed by adding 1 to each component, and $n$ is increased by 1, then $\Fcsm_{\lambda,n}$ gets multiplied by $z_1\cdots z_k$. This means, that---in the $n\geq k+\lambda_1$ range---increasing $n$ by 1, and increasing $\lambda$ by $(1^k)=(1,\ldots,1)$ changes the $GL_k(\C)$-equivariant CSM class of $\Omega_{I_\lambda}$ in a controlled way: in  the Schur expansion the partition of each Schur polynomial is increased by $(1^k)=(1,\ldots,1)$. For example for $k=2$, $\lambda=(3,1)$ we get
\begin{eqnarray*}
\csm_{\beta=0}\left(\Omega_{{\{2,5\}}}\subset \Hom(2,5)\right)
& = &
\rho^{2,0}\left( \Sc_{3,1}+2\Sc_{3,2}+\Sc_{4,1}+3\Sc_{4,2}+3\Sc_{4,3}+2\Sc_{4,4} \right), 
\\
\csm_{\beta=0}\left(\Omega_{{\{3,6\}}}\subset \Hom(2,6)\right)
& = &
\rho^{2,0}\left(\Sc_{4,2}+2\Sc_{4,3}+\Sc_{5,2}+3\Sc_{5,3}+3\Sc_{5,4}+2\Sc_{5,5}\right), 
\\
\csm_{\beta=0}\left(\Omega_{{\{4,7\}}}\subset \Hom(2,7)\right)
& = &
\rho^{2,0}\left(\Sc_{5,3}+2\Sc_{5,4}+\Sc_{6,3}+3\Sc_{6,4}+3\Sc_{6,5}+2\Sc_{6,6}\right),
\end{eqnarray*}
and so on, c.f. `lowering' and `raising' operators in \cite{structure}.
\end{remark}

\section{Generating functions parameterized by partitions} \label{sec:GFpartition}

Corollary \ref{cor:ssmS} claims that $\ssm_{\beta=0}(\Omega_{I_\lambda})$ is obtained applying the substitution $\rho^{k,0}$ to the generating sequence
$\SSS_{z_1,\ldots,z_k} (\Fssm_{\lambda,\infty})$. However this generating sequence changes by adding a 0 to the end of $\lambda$. In this section we present a generating
function independent of such change. This new generating functions depends on infinitely many variables. First, in Section \ref{sec:incrgenfn} we deal with the algebra of
generating functions with infinitely many variables.

\subsection{Increasing the number of variables in generating function}\label{sec:incrgenfn}

Let $h(z_1,\ldots,z_{k+1})$ be a power series in $k+1$ variables. We have that $\SSS_{z_1,\ldots,z_k,z_{k+1}}(h)$ is an infinite linear combination of Schur functions. Some
of the terms correspond to partitions of length at most $k$---call the sum of these terms $\SSS^{\leq k}_{z_1,\ldots,z_{k+1}}(h)$---and the rest corresponds to partitions of
length $k+1$.

\begin{lemma} \label{lem:stab}
Let $f(z_1,\ldots,z_k)$ and $g(z_1,\ldots,z_k,z_{k+1})$ be formal power series such that $g(z_1,$ $\ldots,$ $z_k,0)=1$. Then
\[
\SSS^{\leq k}_{z_1,\ldots,z_{k+1}}(fg) = \SSS_{z_1,\ldots,z_k}(f).
\]
\end{lemma}

\begin{proof}
Let $\prod_{i=1}^k z_i^{a_i}$ and $\prod_{i=1}^{k+1} z_i^{b_i}$ be monomials that occur in $f$ and $g$, respectively,  with non-zero coefficients. Their product
$T=\prod_{i=1}^{k} z_i^{a_i+b_i} \cdot z_{k+1}^{b_{k+1}}$ occurs in $fg$ with non-zero coefficient. From (\ref{eqn:str}) we have that $\SSS_{z_1,\ldots,z_{k+1}}(T)$ is
either 0 or equal to $\pm\SSS_{z_1,\ldots,z_{k+1}} (\prod_{i=1}^{k+1} z_i^{\mu_i})$ where $\mu$ is a partition, and $\mu_{k+1}$ is equal to one of
\[
b_{k+1}, a_k+b_k+1, a_{k-1}+b_{k-1}+2, a_{k-2}+b_{k-2}+3, \ldots, a_1+b_1+k.
\]
Hence, $\SSS^{\leq k}_{z_1,\ldots,z_{k+1}}(T)$ is non-zero, iff $\mu_{k+1}=0$. The listed integers are all necessarily positive except the first one. Hence $\mu_{k+1}=0$ can
only occur if $\mu_{k+1}=b_{k+1}=0$. However, the $g(z_1,$ $\ldots,$ $z_k,0)=1$ condition then implies that the only monomial in $g$ with $b_{k+1}=0$ is the monomial 1. That
is, we have that $b_i=0$ for $i=1,\ldots,k+1$.

We obtained that the only way of obtaining a partition of length at most $k$ in $\SSS^{\leq k}_{z_1,\ldots,z_{k+1}}(fg)$ is by using the constant term 1 of $g$. This proves
the lemma.
\end{proof}

For a function $f(z_1,\ldots,z_k)$ and $N\in \N$ consider
\[
H_N=\SSS_{z_1,\ldots,z_N} \left( f(z_1,\ldots,z_k)
\prod_{j=1}^N \prod_{i=1}^j
\frac{ 1+z_i-z_j}{1+z_i}    \right).
\]
Observe that
\[
H_N=H_{N-1} \cdot \prod_{i=1}^N \frac{1+z_i-z_N}{1+z_i},
\]
and that the factor $\prod_{i=1}^N (1+z_i-z_N)/(1+z_i)$ takes the value 1 if we substitute $z_N=0$. Hence, we can apply Lemma \ref{lem:stab} for $k+1,k+2,\ldots$ and obtain
that the coefficient of $\Sc_\mu$ for any concrete partition $\mu$ stabilizes in $H_N$ as $N\to \infty$. The sum of the stable terms will be denoted by
\[
\SSS_{z_1,z_2,\ldots} \left( f(z_1,\ldots,z_k)
\prod_{j=1}^\infty \prod_{i=1}^j
\frac{ 1+z_i-z_j}{1+z_i} \right).
\]

\subsection{The $\tssm_\lambda$ function}

Recall that a partition is an equivalence class of sequences of weakly decreasing non-negative integers with respect to the equivalence relation generated by
$(\lambda_1,\ldots,\lambda_k)\sim (\lambda_1,\ldots,\lambda_k,0)$. As usual, we will use a representative to denote a partition. We are ready to make a key definition of the
paper.

\begin{definition}
Denote
\[
\tssm_\lambda:=\SSS_{z_1,z_2,\ldots} \left(
\prod_{i=1}^k \left( \frac{z_i}{1+z_i} \right)^{\lambda_i}
\prod_{j=1}^\infty \prod_{i=1}^j
\frac{ 1+z_i-z_j}{1+z_i}
\right).
\]
%As shown, $\tssm_\lambda$ only depends on the partition $\lambda$.
\end{definition}

\begin{example} \rm Some examples are given in the Introduction. Another one is
\begin{align*}
\tssm_{31}=&
\Sc_{31}-(4\Sc_{41}+3\Sc_{32}+3\Sc_{311})+(10\Sc_{51}+13\Sc_{42}+5\Sc_{33}+10\Sc_{321}+6\Sc_{3111}+13\Sc_{411})
\\
& -(20\Sc_{61}+35\Sc_{52}+22\Sc_{43}+35\Sc_{511}+46\Sc_{421}+19\Sc_{331}+10\Sc_{322}+28\Sc_{4111}+22\Sc_{3211})+
\ldots.
\end{align*}
Observe how the partitions that occur in the subscripts grow: Not only the components are larger and larger numbers but the {\em lengths} of the partitions are growing as
well. In the usual picture of Young diagrams the shapes not only ``grow to the right'' but also ``grow downwards''. To our best knowledge this phenomenon is new in algebraic
combinatorics; it does not occur in analogous situations in the theory of equivariant fundamental classes essentially due to \cite[Theorem~2.1]{ts}.
\end{example}

The following conjecture is verified in several special cases.

\begin{conjecture} \label{conj:positivity}
For every partition $\lambda$ the signs in the Schur expansions of $\tssm_{\lambda}$ alternate with the degree. Namely, for a partition $\mu$, $(-1)^{|\mu|-|\lambda|}$ times
the coefficient of $\Sc_\mu$ in $\tssm_{\lambda}$ is non-negative.
\end{conjecture}

\begin{theorem} \label{thm:partition_gen_fv}
  Let $\lambda=(\lambda_1 \geq \ldots \geq \lambda_k)$ and $n\geq \lambda_1+k$. Consider $\Omega_{I_\lambda}\subset \Hom(\C^k,\C^n)$.
  We have
  \[
  \ssm_{\beta=0}(\Omega_{I_{\lambda}})
  =
  \rho^{k,0}(\tssm_\lambda).
  \]
 %
 % The substitution defined by
%\[
%1+c_1t+c_2t^2+\ldots= \frac{1}{\prod_{i=1}^k (1+\alpha_it)}.
%\]
%into the power series $\tssm_{\lambda}$ results  $\ssm_{\beta=0}(\Omega_{I_{\lambda}})$.
 \end{theorem}

\begin{proof}
The statement follows from Corollary \ref{cor:ssmS} if we use the form (\ref{eqn:Fssm2}) for $\Fssm_{\lambda,\infty}$, and Lemma~\ref{lem:stab}.
\end{proof}

The advantage of this theorem compared to Corollary \ref{cor:ssmS} is that now the generator function only depends on the partition. The disadvantage is that this general
generating function has infinitely many variables. The condition $n\geq \lambda_1+k$ is equivalent to the property that
the Young diagram of $\lambda$ fits into a $k\times (n-k)$ rectangle, or to the property that
$I_\lambda\subset \{1,\ldots,n\}$, or to the property that the elements of the orbit on the left hand side
have full rank $k$.

\subsection{SSM classes of general matrix Schubert cells in terms of $\tssm_\lambda$ functions}

Theorem~\ref{thm:partition_gen_fv} gives the $\tssm$-expansion of the $GL_k(\C)$-equivariant SSM classes of ``full-rank'' matrix Schubert cells, that is, those cells in
$\Hom(\C^k,\C^n)$ whose elements have rank $k$. While these cells are of the most interest, we will need $\tssm$-expansions of the SSM classes of smaller rank cells too.

Let $I=\{i_1<\ldots<i_d\}\in \I_{k,n}$ where $|I|=d\leq k$. Recall that the corresponding $\lambda\in \I'_{k,n}$ has the form
\[
\lambda=( \underbrace{n-d,\ldots,n-d}_{k-d}, \lambda_{k-d+1},\lambda_{k-d+2},\ldots,\lambda_k).
\]
Let $\Lambda(I)$ be the set of partitions $\mu$ obtained from this $\lambda$ by weakly increasing only the first $k-d$ components. That is, elements $\mu=(\mu_1,\ldots,\mu_k)$ of $\Lambda(I)$
are partitions and they satisfy
\begin{itemize}
\item $\mu_a\geq n-d$ for $a=1,\ldots,k-d$,
\item $\mu_a=\lambda_a=i_{k+1-a}-(k+1-a)$ for $a=k-d+1,\ldots,k$.
\end{itemize}

\begin{theorem} \label{thm:genOmega}
For $I\in \I_{k,n}$, $|I|=d\leq k$, $\Omega_I\subset \Hom(\C^k,\C^n)$ we have
\[
\ssm_{\beta=0}(\Omega_I)=\rho^{k,0}\left( \sum_{\mu \in \Lambda(I)} \tssm_\mu \right).
\]
\end{theorem}

Observe that Theorem \ref{thm:partition_gen_fv} is the $d=k$ special case of this one.
\begin{proof}
For $N>n$ let $\pi:\Hom(\C^k,\C^N)\to \Hom(\C^k,\C^n)$ be the projection defined by forgetting the bottom $N-n$ rows of a $N\times k$ matrix. The projection $\pi$ is $GL_k(\C)\times B^-_N$-equivariant, where $B^-_N$ acts on the target through the map $B^-_N\to B^-_n$ assigning the upper-left $n\times n$ submatrix to an element of $B^-_N$.

Consider the cylinders
$\pi^{-1}(\Omega_I)\subset \Hom(\C^k,\C^N)$ for $I\in \I_{k,n}$.
%over the $GL_k(\C)\times B^-_n$ orbits of $\Hom(\C^k,\C^n)$.
They are $GL_k(\C)\times B^-_N$-invariant, and from the rank description of orbits (\ref{eqn:rank}) it follows that
%the $\pi^{-1}(\Omega_I)$ sets are unions of the $GL_k(\C)\times B^-_N$ orbits of $\Hom(\C^k,\C^N)$. Namely,
\[
\pi^{-1}(\Omega_I)=\bigcup_{J\in \I_{k,n,N}(I)} \Omega_J
\]
where
\[
\I_{k,n,N}(I)=\{ J\in \I_{k,N}: J \cap \{1,\ldots,n\}=I \}.
\]

The map $\pi$ is a projection, hence it is transversal to the $\Omega_J$ stratification of $\Hom(\C^k,\C^n)$ (a Whitney stratification).
%The projection is $GL_k(\C)\times (B^-_n \times B^-_{N-n})$-invariant, where the action of the factor $B^-_{N-n}$ is trivial on the target.
Hence, (\ref{itrans}) in Section \ref{sec:introCSM} implies that in $GL_k(\C)\times B^-_N$-equivariant cohomology
\begin{align*}
\ssm( \Omega_I\subset \Hom(\C^k,\C^n))  = &
\ssm( \pi^{-1}(\Omega_I)\subset \Hom(\C^k,\C^{N})) \\
 =&
\sum_{J\in \I_{k,n,N}(I)}
\ssm( \Omega_J\subset \Hom(\C^k,\C^N)).
\end{align*}
Because of the $GL_k(\C)\times B^-_N$-action on $\Hom(\C^k,\C^n)$ is through $GL_k(\C)\times B^-_n$, the left hand side can be interpreted as the $GL_k(\C)\times B^-_n$-equivariant SSM class. In particular---while the identity is in the completion of $\C[\alpha_1,\ldots,\alpha_k,\beta_1,\ldots,\beta_N]^{S_k}$---both sides only depend on the variables $\alpha_1,\ldots,\alpha_k,$ $\beta_1,$ $\ldots,$$\beta_n$.

It will be convenient to rewrite the last expression as
\begin{eqnarray*}
\ssm( \Omega_I\subset \Hom(\C^k,\C^n)) &= &
\sum_{J\in \I_{k,n,N(I)}, |J|=k}
\ssm(\Omega_J \subset \Hom(\C^k,\C^N))  \\
& & +
\sum_{J\in \I_{k,n,N(I)}, |J|<k}
\ssm(\Omega_J \subset \Hom(\C^k,\C^N)).
\end{eqnarray*}
Now let us restrict the group action to $GL_k(\C)$---that is we substitute $\beta_i=0$---and apply Theorem~\ref{thm:partition_gen_fv} to the terms in the first summation. We obtain
\begin{eqnarray} \label{eqn:almost}
\ssm_{\beta=0}( \Omega_I\subset \Hom(\C^k,\C^n)) &= &
\sum_{\mu\in \Lambda(I), \mu_1\leq N-k }
\rho^{k,0}(\tssm_\mu)  \\
& & +
\sum_{ |J|<k, J\cap \{1,\ldots,n\}=I}
\ssm_{\beta=0}(\Omega_J \subset \Hom(\C^k,\C^N)). \nonumber
\end{eqnarray}
As $N\to \infty$, the codimensions of the $\Omega_J$'s appearing in the second summation tend to infinity. Hence, the  degree of the second summation tends to infinity (where degree means the degree of the smallest non-zero component). Therefore applying $N \to \infty$ to (\ref{eqn:almost}) proves the theorem.
\end{proof}
\begin{example} \rm
It is instructive to compare the following two examples (c.f. Examples \ref{ex:31}, \ref{ex:31v}):
\begin{eqnarray*}
\ssm_{\beta=0}\left(
\Omega_{\{2,5\}}\subset \Hom(\C^2,\C^5)
\right)
&=&
\rho^{2,0}(\tssm_{(31)}), \\
\ssm_{\beta=0}\left(
\Omega_{\{2\}}\subset \Hom(\C^2,\C^4)
\right)
&=&
\rho^{2,0}(\tssm_{(31)}+\tssm_{(4,1)}+\tssm_{(5,1)}+\ldots).
\end{eqnarray*}
The partition associated to both $\{2,5\}\in \I_{2,5}$ and $\{2\}\in \I_{2,4}$ is $\lambda=(3,1)$, see Section \ref{sec:conventions}. Correspondingly, the fundamental class
of both $\Omega$ orbits above is $\rho^{2,0}(\Sc_{(31)})$. This Schur function is the smallest term of both SSM classes above. However, the full SSM classes are different.
That is, while the fundamental classes of matrix Schubert varieties only depend on the associated partitions (a phenomenon observed in \cite{ss, km}), in SSM theory this
only holds for full rank orbits.
\end{example}

\begin{remark}
Arguments similar to the ones used above (e.g. Theorem \ref{thm:genOmega} for $k,n=\infty$, $I=\emptyset$) show that
\begin{equation}\label{eqn:sum1}
\sum_\lambda \tssm_\lambda=1.
\end{equation}
In fact, using the ``triangularity'' property
\begin{equation}\label{eqn:triangular}
\ssm_\lambda = \Sc_\lambda+\text{higher degree terms}
\end{equation}
we can see that (\ref{eqn:sum1}) is the only linear relation among the functions $\tssm_\lambda$. If we declare $c_i$ ``of order $\varepsilon^i$'', and we declare the Schur
functions $\Sc_\lambda$ ``positive'', then property (\ref{eqn:triangular}) implies that the terms $\tssm_\lambda$ in (\ref{eqn:sum1}) are positive. Hence the collection  $\{
\tssm_\lambda\}_\lambda$ is a (formal power series valued) {\em probability distribution} on the set of partitions. In this language the $GL_\infty$-equivariant SSM class of
an equivariant constructible function on the $GL_\infty \times B^-_\infty$-representation $\Hom(\C^\infty,\C^\infty)$ is the {\em expected value} of the constructible
function. It would be interesting to see applications of this probability theory interpretation in enumertive geometry.
\end{remark}

\section{SSM classes for the $A_2$ quiver representation}\label{sec:A2}

Let $k \leq n$ be non-negative integers, $l=n-k$, and consider the $GL_k(\C)\times GL_{n}(\C)$ representation $\Hom(\C^k,\C^{n})$ defined by $(A,B).\phi=B\circ \phi \circ
A^{-1}$. This representation is also called the $A_2$ quiver representation. The orbits of this representations are
\[
\Sigma^r=\Sigma^r_{k,n}=\{\phi\in \Hom(\C^k,\C^{n}): \dim\ker \phi =r\}
\]
for $r=0,\ldots,k$.
%In this section we will prove the following theorem.

The SSM class of $\Sigma^r_{k,n}$ is a non-homogeneous element in (the completion of)
\[
H^*(B(GL_k(\C)\times GL_n(\C)))=\C[\alpha_1,\ldots,\alpha_k,\beta_1,\ldots,\beta_n]^{S_k\times S_n}.
\]

\begin{theorem}\label{thm:main}
We have
\begin{equation}\label{eqn:cool}
\ssm(\Sigma^r_{k,n})=
\rho^{k,n}\left(
\mathop{\sum_{\lambda_r\geq r+l}}_{\lambda_{r+1}\leq r+l} \tssm_\lambda
\right).
\end{equation}
\end{theorem}

Another way of describing the indexing set in (\ref{eqn:cool}) is that it consists of partitions whose Young diagram is contained in the infinite $\Gamma$-shaped region determined
by an $r\times (r+l)$ rectangle. Equivalently, partitions, whose Young diagram contains the box $(r,r+l)$ but does not contain the box $(r+1,r+l+1)$.

\begin{proof}
First we claim that there exist a formal power series $\PPP^r_l$ in $c_1,c_2,\ldots$ only depending on $r$ and $l$ (not on $k$ and $n$ separately), such that
\[
\ssm(\Sigma^r_{k,n})=
\rho^{k,n}\left( \PPP^r_l\right).
\]
The analogous statement for the fundamental class $[\Sigma^r_{k,n}]$ (or even for Thom polynomials of contact singularities) is well known. In fact, the proof for the
fundamental class depends on the property that the fundamental class of a transversal pullback is the pullback of the fundamental class. Since this property extends from
fundamental class to SSM class, see property (\ref{itrans}) in Section \ref{sec:introCSM}, the same proof also proves our claim. Alternatively, the exclusion-inclusion
formula given in \cite{PP} for $\ssm(\Sigma^r_{k,n})$---recalled in Section \ref{sec:PP}---only depends on $r$ and $l$.

It follows that
\begin{equation}\label{qw1}
\ssm_{\beta=0}(\Sigma^r_{k,n})=
\rho^{k,0}\left( \PPP^r_l\right).
\end{equation}
We have
\[
\Sigma^r_{k,n}=\cup_{I\in \I_{k,n}, |I|=k-r} \Omega_I,
\]
and both the $\Sigma^r_{k,n}$ and the $\Omega_I$ sets are $GL_k(\C)$-invariant. Therefore from Theorem \ref{thm:genOmega} we get
\[
\ssm_{\beta=0}(\Sigma^r_{k,n})=\rho^{k,0}
\left(
\mathop{\sum_{I\in \I_{k,n}}}_{|I|=k-r} \sum_{\lambda\in \Lambda(I)} \tssm_\lambda
\right).
\]
Using the conversion formulas of Section \ref{sec:conventions} this is rewritten as
\begin{equation}\label{qw2}
\ssm_{\beta=0}(\Sigma^r_{k,n})=\rho^{k,0}
\left(
\mathop{\sum_{\lambda=(\lambda_1,\ldots,\lambda_k)}}_{\lambda_r\geq r+l, \lambda_{r+1}\leq r+l}  \tssm_\lambda
\right).
\end{equation}
Comparing (\ref{qw1}) with (\ref{qw2}), and using the fact that
$
\ker \rho^{k,n} =$span$\{ \Sc_\lambda : \lambda_{k+1}\geq n+1\}
$
we obtain
\[ \PPP^r_l=
\mathop{\sum_{\lambda=(\lambda_1,\ldots,\lambda_k)}}_{\lambda_r\geq r+l, \lambda_{r+1}\leq r+l}  \tssm_\lambda
+
\sum_{\lambda_{k+1}\geq 1} a_\lambda \Sc_\lambda.
\]
Since this holds for all $k\geq r$, we have
\[ \PPP^r_l=
\mathop{\sum_{\lambda_r\geq r+l}}_{\lambda_{r+1}\leq r+l}  \tssm_\lambda
\]
what we wanted to prove.
\end{proof}

\begin{remark}
The $A_2$ quiver representation is the prototype of degeneracy loci theory.
The fundamental class of the orbit closures for this representation, $[\Sigma^r_{k,n}]=\rho^{k,n}(\Sc_{(r+l)^r})$ (called Giambelli-Thom-Porteous formula) is a positive and
very simple expansion in terms of the ``atoms'' of fundamental class theory, the Schur functions. The very same positivity and simplicity is displayed in Theorem
\ref{thm:main} for the SSM class, if we choose our ``atoms'' for the SSM theory to be the $\tssm_\lambda$ functions. This is one of the main messages of the present paper:
the natural presentation of SSM classes is in terms of $\tssm_\lambda$ functions. Of course, for more complicated quivers, or for higher jet representations (singularity
theory) the coefficients will be more complicated. We expect, however, that the coefficients will still be non-negative for many geometrically relevant representations. More
evidence towards this expectation will be shown in \cite{Balazs, job}. Finally, this expectation, together with Conjecture \ref{conj:positivity}, is the ``two-step'' positivity
structure we are conjecturing for SSM classes of geometrically relevant degeneracy loci.
\end{remark}

An exclusion-inclusion type formula for $\ssm(\Sigma^r_{k,n})$ was proved by Parusinski-Pragacz \cite{PP}. For completeness, in Appendix A (Section \ref{sec:PP}) we reprove
the Parusinski-Pragacz formula, together with some additional generating series description.

\section{Appendix A: The Parusinski-Pragacz-formula} \label{sec:PP}
A seminal paper on CSM/SSM classes of degeneracy loci is \cite{PP}, where the authors present a sieve type formula for $\ssm(\Sigma^r_{k,n})$ for the $A_2$ quiver
representation (see Section~\ref{sec:A2}). In this section we give a modern proof of their result, but essentially following the line of their arguments. The reason for
giving a new proof is twofold. On the one hand we add to the results of \cite{PP} by giving a generating series description of a main ingredient. On the other hand, in
Section \ref{sec:FiberedRes} we set up a general framework of calculating SSM classes of degeneracy loci once a fibered resolution is found; we believe this will be useful
in future calculations both for quivers and singularities.

For $k\leq n$, and $\mu,\nu$ partitions of length at most $k$ let
\[
D^{k,n}_{\mu,\nu}=\det \begin{pmatrix} \binom{\mu_i+k-i+\nu_j+n-j}{\mu_i+k-i} \end{pmatrix}_{i,j=1,\ldots,k}.
\]

\begin{theorem} [essentially \cite{PP}]\label{thm:PPumbrella}
For $k\leq n$, $l=n-k$, we have
\begin{equation}\label{pp1}
\ssm(\Sigma^r)=\sum_{s=r}^k (-1)^{s-r} \binom{s}{r} \Phi^s_{k,n},
\end{equation}
where
\begin{equation}\label{pp2}
\Phi_{k,n}^s = \SSS_{z_1,\ldots}^{k,n}
\left(  \prod_{i=1}^s \left( \frac{z_i}{1+z_i} \right)^{s+l} \prod_{j=s+1}^\infty \prod_{i=1}^s \frac{1+z_i-z_j}{1+z_i} \right),
\end{equation}
as well as
\begin{equation}\label{pp3}
\Phi^s_{k,n}=\rho^{k,n}\left( \sum_{l(\mu)\leq s} \sum_{l(\nu)\leq s} (-1)^{|\mu|+|\nu|} D^{s,s+l}_{\mu,\nu} \Sc_{(s+l)^s+\mu,\nu^T}\right),
%\qquad \text{\rm (sum for partitions $\mu,\nu$ of length $\leq s$)}
\end{equation}
and
\begin{equation}\label{pp4}
D^{s,s+l}_{\mu,\nu}\geq 0.
\end{equation}
\end{theorem}

\begin{remark}
Statements (\ref{pp1}) and (\ref{pp3}) were proved in \cite{PP} (precisely speaking, Theorem 2.1 of \cite{PP} is for the closure of $\Sigma^r$, but due to additivity of SSM
classes it is obviously equivalent to (\ref{pp1}), cf.~Theorem~ \ref{thm:PP} (2)).
\end{remark}

Statement (\ref{pp4}) is known in relation with Segre classes of tensor products of vector bundles (see \cite{LaLaT}). Here is a sketch of a combinatorial proof. Consider
the oriented graph whose vertices are the integer points of the real plane, and whose edges are all the length 1 segments among them, oriented left/up. Consider the
``source'' points $P_i=(\mu_i+s-i,0)$ for $i=1,\ldots,s$ and the ``sink'' points $Q_j=(0,\nu_j+(s+l)-j)$ for $j=1,\ldots,s$. Applying the Lindstr\"om-Gessel-Viennot lemma
(e.g. \cite{Li}) to this situation interprets $D^{s,s+l}_{\mu,\nu}$ as the number of certain non-intersection paths, that is, $D^{s,s+l}_{\mu,\nu}$ is non-negative.

\smallskip

In the rest of this section---after proving some generalities about fibered resolutions---we give a full proof of Theorem \ref{thm:PPumbrella}. Namely Theorem \ref{thm:PP}
proves (\ref{pp1}), Theorem \ref{thm:PhiS2} proves (\ref{pp2}), and Theorem \ref{thm:Dexp} proves (\ref{pp3}).

\subsection{Fibered resolution} \label{sec:FiberedRes}

Let $\Sigma\subset V$ be in invariant subvariety of the $G$-representation $V$. The $G$-equivariant map $\eta:\tilde{\Sigma} \to V$ is called a fibered resolution of
$\Sigma$, if it is a resolution of singularities of $\Sigma$, moreover, if $\tilde{\Sigma}$ is a total space of a $G$-vector bundle $\tilde{\Sigma}\to K$ over a smooth
compact $G$-variety $K$ and the resolution $\eta$ factors as $\eta=\pi_V\circ i$, where $i: \tilde{\Sigma} \subset K\times V$ is a $G$-equivariant embedding of vector
bundles and $\pi_V$ is the projection to $V$. That is, $\eta: \tilde{\Sigma} \to V$ is a $G$-equivariant fibered resolution, if we have a $G$-equivariant commutative diagram
\[
\xymatrix{ \tilde{\Sigma} \ar[dr] \ar[r]^-i \ar@/^2pc/[rr]^\eta& K\times V \ar[d]^{\pi_K} \ar[r]^-{\pi_V}& V  \\
                              & K. &                                          }
\]
Consider the $G$-equivariant quotient bundle $\nu=(K \times V \to K)/(\tilde{\Sigma}\to K)$ over $K$. This bundle, pulled back to $\tilde{\Sigma}$ is the normal bundle of
the embedding $i:\tilde{\Sigma} \to K\times V$. The bundle $\nu$ pulled back over $K \times V$ has a natural section $\sigma$ given by $\sigma(k,v)=v+i(\tilde{\Sigma}_k)$
(where $k\in K, v\in V$ and $\tilde{\Sigma}_k$ is the fiber of $\tilde{\Sigma} \to K$ over $k$). The section $\sigma$ is transversal to the 0-section, and
$\sigma^{-1}(0)=i(\tilde{\Sigma})$. Hence we have the remarkable situation that the normal bundle of $i(\tilde{\Sigma})\subset K\times V$ extends to a bundle $\nu$ over
$K\times V$.

In the cohomology calculations below we work in $G$-equivariant cohomology, and---as custo\-ma\-ry---we do not indicate pull-back bundles (e.g. $\nu$ may denote bundles over
$\tilde{\Sigma}, K$, or $K\times V$ respectively). The cohomology of a total space and the base space of a vector bundle will be identified without explicit notation.

We will be concerned with two $G$-equivariant cohomology classes in $V$: the fundamental class $[\Sigma]$ of $\Sigma$ in $V$, and the common value
\begin{equation}\label{eqn:Phi}
\Phi_\Sigma:=
\frac{\eta_*(c(T\tilde{\Sigma}))}{c(V)}
=
\eta_*\left(  \frac{c(T\tilde{\Sigma})}{c(V)} \right)
=
\eta_*\left( c(-\nu) c(TK) \right).
\end{equation}
The fist equality follows by adjunction and second one follows from the calculation
\[
\frac{c(T\tilde{\Sigma})}{c(V)}=\frac{c(T\tilde{\Sigma})}{c(V)c(TK)}c(TK)=c(-\nu)c(TK).
\]

\begin{proposition} \label{prop:resH}
We have
\[
[\Sigma]  = \int_K e(\nu), \qquad\qquad
\Phi_\Sigma  = \int_K e(\nu) c(-\nu) c(TK).
\]
\end{proposition}

\begin{proof}
Since the section $\sigma$ described above is transversal to the 0-section, and $\sigma^{-1}(0)=i(\tilde{\Sigma})$ we have $i_*(1)=e(\nu)$ and
\[
[\Sigma]=\eta_*(1)=\pi_{V*}i_*(1)=\pi_{V*}(e(\nu))=\int_K e(\nu)
\]
proving the first statement. The second statement follows from the calculation
\[
\Phi_\Sigma
=\eta_*( c(-\nu)c(TK) )
=\int_K i_*(c(-\nu)c(TK))
=\int_K i_*(c(-\nu))c(TK) \hskip 3 true cm \
\]
\[
\ \hskip 4 true cm =\int_K i_*(i^*c(-\nu)) c(TK)
=\int_K i_*(1) c(-\nu) c(TK)
=\int_K e(\nu) c(-\nu) c(TK),
\]
where we (repeatedly) used the adjunction formula, and the fact that $\nu$ extends from $\tilde{\Sigma}$ to~$K\times V$.
\end{proof}

Only the second statement of Proposition \ref{prop:resH} is relevant for the present paper---and in fact the first one follows from the second one. We included the first one
(well known in the theory of fundamental classes \cite{bsz, noAss, ts}) for comparison purposes.

%\begin{remark} \rm
%An ingredient of the calculation in the proof is $i_*(c(-\nu))$. Observe that this expression is equal to %$\ssm(\tilde{\Sigma}\subset K\times V)$.
%\end{remark}

\subsection{CSM/SSM classes in terms of $\Phi$-classes}
 Consider the fibered resolution
\begin{equation}\label{fig:PPP}
\xymatrix{ \tilde{\Sigma^r} \ar[dr] \ar[r]^-i \ar@/^2pc/[rr]^{\eta_r}& \Gr_r(\C^k) \times \Hom(\C^k,\C^{n}) \ar[d]^{\pi_1} \ar[r]^-{\pi_2}& \Hom(\C^k,\C^{n})  \\
                              & \Gr_r(\C^k) &                                          }
\end{equation}
of $\overline{\Sigma^r}$, where
\[
\tilde{\Sigma^r}=\{ (V,\phi)\in \Gr_r(\C^k) \times \Hom(\C^k,\C^{n}): \phi|_V=0\}
\]
with the obvious embedding into $\Gr_r(\C^k) \times \Hom(\C^k,\C^{n})$ and projection to $\Gr_r(\C^k)$.

%The first statement of Proposition \ref{prop:resH} gives the well know formula $[\overline{\Sigma^r}]=\int_{\Gr_r(\C^k)} e(-\nu)$, which can be used to prove the
%Giambelli-Thom-Porteous formula.

Define $\Phi_{k,n}^{r}$ to be the class in (\ref{eqn:Phi}) for the fibered resolution (\ref{fig:PPP}).
%Now we show how
%the second formula of Proposition~\ref{prop:resH} plays a role in the CSM/SSM class calculation of $\Sigma^r$.
The following theorem is  equivalent to Theorem 2.1 of \cite{PP}.

\begin{theorem} \label{thm:PP}
We have
\begin{enumerate}
\item \label{PPthmegy}
\[
\Phi_{k,n}^r
=\sum_{s=r}^k \binom{s}{r} \ssm(\Sigma^s)
=\sum_{s=r}^k \binom{s-1}{r-1} \ssm(\overline{\Sigma^s}).
\]
\item \label{PPthmket}
\[
\ssm(\overline{\Sigma^r})=\sum_{s=r}^k (-1)^{s-r} \binom{s-1}{r-1} \Phi_{k,n}^s,
\qquad
\ssm(\Sigma^r)=\sum_{s=r}^k (-1)^{s-r} \binom{s}{r} \Phi_{k,n}^s.
\]
\end{enumerate}
\end{theorem}

\begin{proof}
The preimage at $\eta_r$ of a point in $\Sigma^s$ ($r\leq s \leq k$) is $\Gr_r(\C^s)$ whose Euler characteristic is~$\binom{s}{r}$. Hence property (\ref{ielore}) from
Section \ref{sec:introCSM} for $\eta_r$ implies
\[
\eta_{r!}( c(T\tilde{\Sigma^r}) ) = \sum_{s=r}^k \binom{s}{r} \csm(\Sigma^s).
\]
Dividing both sides by $c(\Hom(\C^k,\C^{n}))$ we obtain $\Phi_{k,n}^r$ on the left hand side, and $\sum_{s=r}^k \binom{s}{r} \ssm(\Sigma^s)$ on the right hand side, which
proves the first equality in part (\ref{PPthmegy}).
Using the additivity property of CSM classes, (\ref{iadd}) from Section \ref{sec:introCSM}, we obtain
\[
\sum_{s=r}^k \binom{s}{r} \csm(\Sigma^s)=
\sum_{s=r}^k \binom{s}{r} ( \csm(\overline{\Sigma^s}) - \csm(\overline{\Sigma^{s+1}}))=
\hskip 4 true cm \ \]
\[
\ \hskip 4 true cm
\sum_{s=r}^k \left(\binom{s}{r}-\binom{s-1}{r}\right) \csm(\overline{\Sigma^s})=
\sum_{s=r}^k \binom{s-1}{r-1} \csm(\overline{\Sigma^s}).
\]
Dividing by $c(\Hom(\C^k,\C^{n}))$ proves the second equality in part (\ref{PPthmegy}).

Part (\ref{PPthmket}) of the theorem is the algebraic consequence of part (\ref{PPthmegy}); it follows from the fact that the inverse of the Pascal matrix $\begin{pmatrix}
\binom{s}{r} \end{pmatrix}_{s,r}$ is the matrix $\begin{pmatrix} (-1)^{s-r} \binom{s}{r} \end{pmatrix}_{s,r}$, see e.g. \cite{CallVelleman}.
\end{proof}

\subsection{Formulas for $\Phi$-classes}
Let $\alpha_u$, $u=1,\ldots,k$ and $\beta_v$, $v=1,\ldots,n$ denote the Chern roots of $GL_k(\C)$ and $GL_{n}(\C)$ respectively. Then
\[
\Phi_{k,n}^s \in \C[\alpha_1,\ldots,\alpha_k,\beta_1,\ldots,\beta_{n}]^{S_k\times S_{n}}
\]
Denoting the Chern roots of the tautological subbundle over $\Gr_s(\C^k)$ by $\gamma_1,\ldots,
\gamma_s$, and the Chern roots of the tautological quotient bundle by $\delta_1,\ldots, \delta_{k-s}$  Proposition \ref{prop:resH} implies
\begin{equation}\label{eq:intPhi}
\Phi^s_{k,n}(\alpha_1,\ldots,\alpha_k;\beta_1,\ldots,\beta_n)=
\int_{\Gr_s(\C^k)} \prod_{i=1}^s \prod_{v=1}^{n} \frac{\beta_v-\gamma_i}{1+\beta_v-\gamma_i}
\prod_{i=1}^s\prod_{j=1}^{k-s} (1+\delta_j-\gamma_i).
\end{equation}

First let us calculate a special case, $s=k$. We have
\[
\Phi^s_{s,s+l}=\prod_{u=1}^s\prod_{v=1}^{s+l}\frac{\beta_v-\alpha_u}{1+\beta_v-\alpha_u} =
\prod_{u=1}^s\prod_{v=1}^{s+l} (\beta_v-\alpha_u)
\sum_{l(\mu)\leq s} \sum_{l(\nu)\leq s} (-1)^{|\mu|+|\nu|} D_{\mu,\nu}^{s,s+l}
\rho^{s,0}(\Sc_\mu)
\rho^{0,s+l}(\Sc_{\nu^T}).
\]
Here we used the Schur function expansion of ``Segre classes of a tensor product'' from \cite{LaLaT}.
Using the ``factorization formula'' of Schur functions we obtain
\begin{equation} \label{eqn:specPhi}
\Phi^s_{s,s+l} =
\rho^{s,s+1}
\left(
\sum_{l(\mu)\leq s} \sum_{l(\nu)\leq s} (-1)^{|\mu|+|\nu|} D_{\mu,\nu}^{s,s+l}
\Sc_{(s+l)^s+\mu,\nu^T}\right).
\end{equation}

\begin{lemma}[Supersymmetry lemma] \label{lem:ss}
We have
\[
\Phi^s_{k+1,n+1}(\alpha_1,\ldots,\alpha_k,t;\beta_1,\ldots,\beta_n,t)=
\Phi^s_{k,n}(\alpha_1,\ldots,\alpha_k;\beta_1,\ldots,\beta_n).
\]
\end{lemma}

\begin{proof}
The equality follows from interpreting both sides with equivariant localization
\[
\Phi^s_{k,n}=
\sum_I
\prod_{u\in I} \prod_{v=1}^{n} \frac{\beta_v-\alpha_u}{1+\beta_v-\alpha_u}
\prod_{u\in I} \prod_{w\in \bar{I}} \frac{1+\alpha_w-\alpha_u}{\alpha_w-\alpha_u}
\]
where the summation runs for $s$-elments subsets $I$ of $\{1,\ldots,k\}$, and $\bar{I}=\{1,\ldots,k\}-I$.
\end{proof}

%The goal in this section is to give better formulas for $\Phi_{k,n}^s$.

\begin{definition} For $k\leq n$ non-negative integers, $l=n-k$, and $1\leq s \leq k$ define
\[
\F^s_{k,n}=\prod_{i=1}^s \left( \frac{z_i}{1+z_i} \right)^{s+n-k}
 \prod_{j=s+1}^k\prod_{i=1}^s\frac{1+z_i-z_j}{1+z_i},
\]
\[
\F^s_{l}=\F^s_{\infty,\infty+l}=\prod_{i=1}^s \left( \frac{z_i}{1+z_i} \right)^{s+l}
 \prod_{j=s+1}^\infty\prod_{i=1}^s\frac{1+z_i-z_j}{1+z_i}.
\]
\end{definition}

\begin{proposition} \label{prop:PhiS1}
For $k\leq n$, $l=n-k$, and $1\leq s \leq k$ we have
\[
\Phi_{k,n,\beta=0}^s =\SSS^{k,0}_{z_1,\ldots,z_k}\left( \F^s_{k,n} \right)
       =\SSS^{k,0}_{z_1,\ldots}\left( \F^s_{l} \right).
\]
\end{proposition}

\begin{proof}
Let
\[
U=\prod_{u=1}^s \left(\frac{-\alpha_u}{1-\alpha_u}\right)^n
\prod_{w=s+1}^k \prod_{u=1}^s \frac{1+\alpha_w-\alpha_u}{\alpha_w-\alpha_u}.
\]
Substituting $\beta_i=0$ for $i=1,\ldots,n$ in (\ref{eq:intPhi}) the equivariant localization formula for the integral yields
\[
\Phi_{k,n,\beta=0}^s=\sum_{\sigma\in S_k/S_s\times S_{k-s}}  U(\sigma(\alpha_1,\ldots,\alpha_k)).
\]
For
\[
V=U\cdot
\left(\prod_{u=1}^s (-\alpha_u)^{s-u}\prod_{w=1}^s\prod_{u=1}^{w-1}\frac{1}{\alpha_w-\alpha_u}\right)
\left(\prod_{u=s+1}^k (-\alpha_u)^{k-u}\prod_{w=s+1}^k\prod_{u=s+1}^{w-1}\frac{1}{\alpha_w-\alpha_u}\right)
\]
we have
\[
\sum_{\sigma\in S_s \times S_{k-s}} V(\sigma(\alpha_1,\ldots,\alpha_k))=U,
\]
and hence
\begin{equation} \label{eq:po1}
\Phi^s_{k,n,\beta=0}=\sum_{\sigma\in S_k} V(\sigma(\alpha_1,\ldots,\alpha_k)).
\end{equation}
Observe that
\begin{equation}\label{eq:po2}
V(\sigma(\alpha_1,\ldots,\alpha_k))
=\Res{z_k=-\alpha_{\sigma(k)}}\ldots \Res{z_2=-\alpha_{\sigma(2)}} \Res{z_1=-\alpha_{\sigma(1)}} f
\end{equation}
for
\begin{align*}
f&=
\prod_{i=1}^s \left( \frac{z_i}{1+z_i}\right)^n \prod_{j=s+1}^k\prod_{i=1}^s (1+z_i-z_j) \prod_{i=1}^s z_i^{s-i} \prod_{i=s+1}^k z_i^{k-i}
\frac{\prod_{1\leq j \leq i \leq k} (z_i-z_j)}{\prod_{j=1}^k \prod_{u=1}^k (z_j +\alpha_u)} \\
 &=
F_{k,n}^s \cdot \prod_{1\leq i < j \leq k} \left( 1-\frac{z_i}{z_j} \right)
\prod_{i=1}^k \frac{1}{\prod_{u=1}^k (1+\alpha_u/z_i)} \prod_{i=1}^k \frac{dz_i}{z_i}.
\end{align*}
The only non-zero finite residues of the form $f$ are the ones on the right hand side of (\ref{eq:po2})---remember that the $1/(1+z_i)$ factors are just abbreviations of the
formal series $1-z_1+z_1^2-\ldots$. Hence from (\ref{eq:po1}) and (\ref{eq:po2}), using the Residue Theorem, we obtain that
\[
\Phi_{k,n,\beta=0}^s=(-1)^k \RES_k \left( f\right),
\]
which, using (\ref{eqn:SchurRES2}) yields the first equality of the proposition.

The second equality follows from the first one and Lemma \ref{lem:stab}.
\end{proof}

\begin{theorem} \label{thm:PhiS2}
For $k\leq n$, $l=n-k$, and $1\leq s \leq k$ we have
\[
\Phi_{k,n}^s =\SSS^{k,n}_{z_1,\ldots}\left( \F^s_l \right).
\]
\end{theorem}

\begin{proof}
Let $N$ be a non-negative integer and consider $\Phi^s_{k+N,n+N}$. Suppersymmetry (Lemma \ref{lem:ss}) implies that this can be written as a linear combination of Schur
functions
$\rho^{k+N,n+N}(\Sc_\lambda)$. Recall from Section \ref{sec:incrgenfn} that for such linear combinations $f$, the notation $f^{\leq m}$ is meant to be the sum of terms
corresponding to partition with length at most $m$.
By Proposition \ref{prop:PhiS1}
\[\Phi^s_{k+N,n+N}(\alpha_1,\ldots,\alpha_{k+N};0,\ldots,0)=\SSS_{z_1,\ldots}^{k+N,0}(\F_l^s).
\]
Since $\rho^{k+N,0}(\Sc_\lambda)=0$ if and only if $\ell(\lambda)>k+N$ we have that
\[
\left(
\Phi_{k+N,n+N}^s
\right)^{\leq k+N} = \left( \SSS_{z_1,\ldots}^{k+N,n+N}(\F_l^s) \right)^{\leq k+N}.
\]
Substituting $\alpha_{k+1}=\alpha_{k+2}=\ldots=\alpha_{k+N}=\beta_{n+1}=\beta_{n+2}=\ldots=\beta_{n+N}=0$, and using the Suppersymmetry Lemma \ref{lem:ss} we get
\[
\left( \Phi_{k,n}^s \right)^{\leq k+N} = \left( \SSS_{z_1,\ldots}^{k,n}( \F_l^s) \right)^{\leq k+N}.
\]
Since this holds for any $N$, the proof is complete.
\end{proof}

\begin{theorem} \label{thm:Dexp}
 For $k\leq n$, $l=n-k$, $0\leq s \leq k$ we have
\begin{equation}
\Phi^s_{k,n}=\rho^{k,n}\left( \sum_{l(\mu)\leq s} \sum_{l(\nu)\leq s} (-1)^{|\mu|+|\nu|} D^{s,s+l}_{\mu,\nu} \Sc_{(s+l)^s+\mu,\nu^T}\right),
%\qquad \text{\rm (sum for partitions $\mu,\nu$ of length $\leq s$)}
\end{equation}
\end{theorem}

\begin{proof}
By the Supersymmetry Lemma \ref{lem:ss} we know that
\begin{equation}\label{eqn:ip1}
\Phi^s_{k,n}=\rho^{k,n}\left( \sum d_\lambda \Sc_{\lambda} \right).
\end{equation}
First we claim that if $\lambda_{s+1} > s $ then $d_\lambda=0$.

According to Theorem \ref{thm:PhiS2} we have $\Phi^s_{k,n}=\SSS_{z_1,\ldots}^{k,n}(\F^s_l)$. The monomials occurring in the Taylor expansion of $\F^s_l$ are of the form
\[
z_1^{a_1}\ldots z_s^{a_s} z_{s+1}^{\varepsilon_1} z_{s+2}^{\varepsilon_2} \ldots z_{s+q}^{\varepsilon_q},
\]
with all $\varepsilon_i \in \{0,1,\ldots,s\}$. Hence $\Phi^s_{k,n}$ is a sum of possibly fake Schur functions $\Sc_\lambda$ satisfying
\begin{equation} \label{eqn:partition_property}
i>s \Rightarrow \lambda_i \leq s.
\end{equation}
Observe that property (\ref{eqn:partition_property}) does not change if one applies the straightening laws (\ref{eqn:str}). Hence $\Phi^s_{k,n}$ is also the sum of Schur
functions satisfying (\ref{eqn:partition_property}). For partitions this property is equivalent to $\lambda_{s+1}\leq s$.

We can hence improve (\ref{eqn:ip1}), and write
\begin{equation}\label{eqn:ip2}
\Phi^s_{k,n}=\rho^{k,n}\left( \sum_{\lambda_s\leq s} d_\lambda \Sc_{\lambda} \right).
\end{equation}
Let us substitute $\alpha_{s+1}=\alpha_{s+1}=\ldots=\alpha_k=\beta_{s+l+1}=\beta_{s+l+2}=\ldots=\beta_n=0$ in (\ref{eqn:ip2}). According to the Supersymmetry Lemma
\ref{lem:ss} we obtain
\begin{equation}\label{eqn:ip3}
\Phi^s_{s,s+l}=\rho^{s,s+l}\left( \sum_{\lambda_s\leq s} d_\lambda \Sc_\lambda\right).
\end{equation}
Observe that for a $\lambda$ with $\lambda_s\leq s$ the Schur function $\rho^{k,n}(\Sc_\lambda)$ is not 0. Hence each $d_\lambda$ in (\ref{eqn:ip3}) has to be the value
described in (\ref{eqn:specPhi}). This proves the theorem.
\end{proof}

\section{Appendix B: Comparing CSM and SSM classes of Schubert and matrix Schubert cells}\label{appB}

In this section we summarize the localization and residue formulas for both equivariant CSM and SSM classes of both matrix Schubert cells and Schubert cells.

Matrix Schubert cells are subsets of $\Hom(\C^k,\C^n)$, and their $GL_k\times B^-_n$-equivariant CSM and SSM classes are elements of
$\C[\alpha_1,\ldots,\alpha_k,\beta_1,\ldots,\beta_n]^{S_k}$ (and its completion). Schubert cells are subsets of the Grassmanniann $\Gr_k(\C^n)$, and their
$B_n^-$-equivariant CSM and SSM classes are elements of a {\em quotient} ring of $\C[\alpha_1,\ldots,\alpha_k,\beta_1,\ldots,\beta_n]^{S_k}$.

For matrix Schubert cells let us now restrict our attention to the ones whose elements have full rank $k$. Then
both versions of Schubert cells can be parameterized by partitions $\lambda=(\lambda_1\geq \ldots \geq \lambda_k)$ with $\lambda_1\leq n-k$ and $\lambda_k\geq 0$, or
equivalently, by subsets $\{i_1<i_2<\ldots<i_k\}$ of $\{1,\ldots,n\}$. The transition between the two parameters is $i_a=\lambda_{k+1-a}+a$.

Below we give two formulas for CSM and SSM classes of matrix and ordinary Schubert cells. The first one is the formula for the appropriate class of the given Schubert cell.
The second formula is the Schur polynomial expansion of the $\beta_v=0$ substitution. In matrix Schubert settings this means $GL_k(\C)$-equivariant formulas, and in the
Grassmannian settings this means non-equivariant formulas. Denote $\sym f=\sum_{\sigma\in S_k} f(\alpha_{\sigma(1)},\ldots,\alpha_{\sigma(k)})$.

\begin{theorem} We have the following formulas for matrix and ordinary Schubert cells.
\begin{enumerate}
\item CSM class of a matrix Schubert cell:
\[
\sym
\prod_{u=1}^k
\left( \prod_{v=i_u+1}^n (1+\beta_v-\alpha_u)
 \prod_{v=1}^{i_u-1} (\beta_v-\alpha_u)
 \prod_{v=u+1}^k \frac{1+\alpha_u-\alpha_v}{\alpha_u-\alpha_v}
\right),
\]
\[
\SSS^{k,0}_{\zz}\left(
\prod_{j=1}^k z_j^{\lambda_j} \prod_{j=1}^k (1+z_j)^{n-i_{k+1-j}}
\prod_{1\leq i< j \leq k}(1+z_i-z_j)
\right);
\]
\item
CSM class of a Schubert cell:
\[
\sym
\prod_{u=1}^k
\left( \prod_{v=i_u+1}^n (1+\beta_v-\alpha_u)
 \prod_{v=1}^{i_u-1} (\beta_v-\alpha_u)
 \prod_{v=u+1}^k \frac{1}{(\alpha_u-\alpha_v)(1+\alpha_v-\alpha_u)}
\right),
\]
\[
\SSS^{k,0}_{\zz}\left(
\prod_{j=1}^k z_j^{\lambda_j} \prod_{j=1}^k (1+z_j)^{n-i_{k+1-j}}
\frac{1}{\prod_{1\leq j<i \leq k}(1+z_i-z_j) }
\right);
\]
\item SSM class of a matrix Schubert cell:
\[
\sym
\prod_{u=1}^k
\left(
\frac{1}{1+\beta_{i_u}-\alpha_{u}}
 \prod_{v=1}^{i_u-1} \frac{ \beta_v-\alpha_u}{1+\beta_v-\alpha_u}
 \prod_{v=u+1}^k \frac{1+\alpha_u-\alpha_v}{\alpha_u-\alpha_v}
\right),
\]
\[
\SSS^{k,0}_{\zz}\left(
\prod_{j=1}^k \left(\frac{z_j}{1+z_j}\right)^{\lambda_j} \prod_{j=1}^{k}  \prod_{i=1}^j \frac{1+z_i-z_j}{1+z_i}
\right).
\]
\item The SSM classes of Schubert cells are represented by SSM classes of matrix Schubert cells (hence formulas of (3) are representatives of these SSM classes).
\end{enumerate}
\end{theorem}

\begin{proof}
Formulas in (1) and (3) are in this paper (Theorems \ref{thm:csmW}, \ref{thm:genWD}, Corollary \ref{cor:osztva}, Theorem \ref{thm:genWD}), (4) follows from (3) via Theorem
\ref{thm:ohmoto}. Formulas in (2) follow from (4)---or can be deduced from results in \cite{AM1, RV}.
\end{proof}

Let us comment on the two positivity results/conjectures known about these classes. One is the result of Huh \cite{huh} (conjectured earlier by Aluffi and Mihalcea
\cite{AM1}): the Schur expansion in (2) has non-negative coefficients. The other is our Conjecture \ref{conj:positivity}, that the Schur expansion in (3) has alternating
signs.

We are not aware of any connection between the two positivity properties. One fact which makes the comparison difficult is that the three-term factors $1+z_i-z_j$ are in the
numerator and denominator respectively in the two cases. Another key difference is that {\em not all} the infinitely many coefficients of the generating sequence in (2) are
positive, only the ones corresponding to partitions $\subset (n-k)^k$---which fact does not contradict to Huh's theorem since the Schur functions corresponding to the other
partitions are 0 in the quotient ring. However, our Conjecture \ref{conj:positivity} is about {\em all} the Schur coefficients of the series in (3)---even if $k\to \infty$.

\end{document}